\documentclass[11pt]{amsart}
\usepackage[margin=1in]{geometry}
\usepackage{xcolor}
\numberwithin{equation}{section}
\newtheorem{thm}{Theorem}[section]

\newtheorem{lem}[thm]{Lemma}

\newtheorem{defn}[thm]{Definition}

\newtheorem{clm}[thm]{Claim}

\newcommand*{\avint}{\mathop{\ooalign{$\int$\cr$-$}}}
\newcommand{\nua}{|\nabla u|}
\newcommand{\br}{B_R(x_0)}
\newcommand{\brh}{B_{\frac{R}{2}}(x_0)}
\newcommand{\brn}{B_{R_n}(x_0)}
\newcommand{\brno}{B_{R_{n+1}}(x_0)}

\newcommand{\ra}{\rightarrow}
\newcommand{\ve}{\varepsilon}
\newcommand{\at}{a_\tau}

\newcommand{\mdiv}{\textup{div}}
\newcommand{\ut}{u_\tau}

\newcommand{\vt}{v_\tau}

\newcommand{\ft}{g_\tau}

\newcommand{\rt}{\rho_\tau}

\newcommand{\RN}{\mathbb{R}^N}

\newcommand{\dest}{\Delta\rho }

\newcommand{\io}{\int_\Omega}

\newcommand{\po}{\partial\Omega }

\newcommand{\plap}{\textup{div}\left(\partial_z E(\nabla u)\right)}

\newcommand{\nnu}{|\nabla u|}

\begin{document}
	\title[An exponential PDE in crystal surface models]{Partial regularity for an exponential PDE  in crystal surface models}
	\author{Xiangsheng Xu}\thanks
	{Department of Mathematics and Statistics, Mississippi State
		University, Mississippi State, MS 39762.
		{\it Email}: xxu@math.msstate.edu. {\it Nonlinearity}, to appear.} 
	\begin{abstract} We study regularity properties of weak solutions to the boundary value problem for the equation $-\Delta \rho +a u=f$ in a bounded domain $\Omega\subset \mathbb{R}^N$, where $\rho=e^{-\mbox{div}\left(|\nabla u|^{p-2}\nabla u+\beta_0|\nabla u|^{-1}\nabla u\right)}$. This problem is derived from the mathematical modeling of crystal surfaces. It is known that the exponential term can be a measure-valued function. In this paper we obtain a partial regularity result, which asserts that there exists an open subset $\Omega_0\subset \Omega$ such that $|\Omega\setminus\Omega_0|=0$ and the exponential term is locally bounded in $\Omega_0$. Furthermore, if $x_0\in \Omega\setminus\Omega_0$, then $\rho$ vanishes of $N+2-\varepsilon$ order at $x_0$ for each $\varepsilon\in(0,2)$. 	
		\end{abstract}
	\maketitle

	\section{Introduction}
	Let $\Omega$ be a bounded domain in $\RN$ with smooth boundary $\po$ and $\nu$ the unit outward normal to $\po$. In this paper we consider the boundary value problem
	\begin{eqnarray}
		-\Delta e^{-\plap} +a u&\ni&f\ \ \mbox{in $\Omega$,}\label{sta1}\\
		\nabla e^{-\plap}\cdot\nu&\ni&0\ \ \mbox{on $\po$,}\label{sta2}\\
		\nabla u\cdot\nu&=&0\ \ \mbox{on $\po$}\label{sta3}
	\end{eqnarray}
for given data $  E(\nabla u), a$, and $f$ with properties:	  
\begin{enumerate}
\item[(H1)]The function $E=E(z)$ is given by
$$
		E(z) =\frac{1}{p}|z|^p+\beta_0 |z|,\ \ z\in \mathbb{R^N}, \ p>1, \ \beta_0>0,
$$
	and  $\partial_zE(z)$ denotes its subdifferential (\cite{H}, p.49);
\item[(H2)]	$a\in (0, \infty), \ f\in W^{1,p}(\Omega)\cap L^\infty(\Omega)$.
\end{enumerate}
Recall that the subdifferential $\partial_zE$ of $E$ is the multi-valued function from $\RN$ to $\RN$ defined by
$$x\in \partial_zE(z)\Longleftrightarrow E(y)\geq E(z)+x\cdot (y-z)\ \ \mbox{for each $y\in\RN$.}$$
A simple calculation yields
\begin{equation}
	\partial_zE(z)=\left\{\begin{array}{cc}
		|z|^{p-2}z+\beta_0 |z|^{-1}z&\mbox{if $z\ne 0$,}\\
		\beta_0{[-1,1]}^{N}&\mbox{if $z=0$,}
	\end{array}\right.\nonumber
\end{equation}
which explains the inclusion sign ``$\ni$'' in \eqref{sta1}-\eqref{sta2}.

Our interest in this
problem originated in the mathematical modeling of crystal surface growth. The manufacturing of crystal films lies at the heart of modern nanotechnology. How to accurately predict the motion of a crystal surface is of fundamental importance. Many continuum  models have been developed for this purpose \cite{AKW,G1,GLL2,GLL3,GLLM,GLLX,X1,X2}. They are often obtained as the continuum limit of a family of kinetic Monte Carlo models of crystal surface relaxation that includes both the solid-on -solid and discrete Gaussian models. In this paper we consider the model developed in \cite{GLL3}. In this model,  $u$ is the surface height.  It is now well-established that the continuum relaxation of a crystal surface below the roughing temperature is governed by the conservation law 
\begin{equation}\label{r31}
	\partial_t u + \mbox{div} J = 0,
\end{equation}
where $J$ is the adatom flux. Denote by $D$ the mobility and $\Gamma_s$ the local equilibrium density of adatoms. Then Fick's law \cite{MK} asserts
$$J = -D\nabla\Gamma_s.$$
In the diffusion-limited (DL) regime, where the dynamics is dominated by the diffusion across the
terraces, we may take
$$D\equiv1.$$
An expression for $\Gamma_s$ can be inferred from 
the Gibbs-Thomson relation \cite{KMCS,RW,MK} to be
$$
\Gamma_s=\rho_0e^{\frac{\mu}{kT_s}},$$
where $\mu$ is the chemical potential, $\rho_0$ is the constant reference density, $T_s$ is the temperature, and $k$ is the Boltzmann constant. 
According to \cite{CRSC}, we can take the general surface energy $G(u)$ to be
\begin{eqnarray*}
	G(u)= \io E(\nabla u)dx=\frac{1}{p}\io|\nabla u|^pdz+\beta_0\io|\nabla u|dz.
\end{eqnarray*}
Here $E(z)$ is given as in (H1).
As observed in \cite{MW},  this type of energy forms can retain many interesting features of the microscopic system that are lost in the more standard scaling regime. 
The chemical potential $\mu$ is
defined as the change per atom in the surface energy. This means
\begin{equation*}
	\mu=\frac{\delta G}{\delta u}=-\mbox{div}\left(\nua^{p-2}\nabla u+\beta_0\frac{\nabla u}{\nua}\right)=-\plap.
\end{equation*}
After incorporating some physical parameters into the scaling of the time and/or spatial variables, we obtain from \eqref{r31} that
\begin{equation}\label{exp}
	\partial_t u\in \Delta e^{-\mbox{div}\left(\partial_z E(\nabla u)\right)}.
\end{equation}
Our equation \eqref{sta1} is obtained by discretizing the time derivative in the above equation.

Continuum models of this type are phenomenological in nature. That is, they are derived from empirical data and observed phenomena, not first principles. Hence, their mathematical validation is important. Unfortunately, current analytical results are still far-lacking.
For example, the existence assertion for \eqref{exp} coupled with initial boundary conditions is still open. The main mathematical challenge is the exponential non-linearity involved. The function $e^s$ decays rapidly to $0$ as $s\rightarrow-\infty$. Thus, it is extremely difficult to derive any estimates for the exponential term near $-\infty$. The authors in \cite{LX,GLL3,GLLX,X1} circumvented this issue by allowing the possibility that the exponential term be a measure. In fact, an explicit solution obtained in \cite{LX} showed that this possibility did occur. We refer the reader to the remark following Theorem 1.1 in \cite{LX} for details.

The key difference between our work here and the existing papers is that we allow the exponent term to include
	the 1-Laplacian,
	  which results in more accurate models \cite{GLL3} and also gives rise to new mathematical challenges. 
	  In almost all the previous papers one either linearizes the exponential function as has been done in \cite{GG,GK}, or takes the exponent to be the Laplacian as has been done in \cite{GLL3,LX}. In \cite{X1}, the author took the exponent to be the $p$-Laplacian with $p>1$ and obtained the existence of a weak solution. Here we generalize the exponent to include the 1-Laplacian. It turns out that the existence assertion in \cite{X1} still holds. More importantly, we develop a
	  partial regularity theorem for our weak solutions.
	   We say that a point $x_0\in \Omega$ is regular if there is a neighborhood of $x_0$ in which $\nabla u$ and the entire exponential term are both bounded functions. The regular set, the collection of all regular points, is always open. Our main result asserts that its complement, the so-called singular set, has Lebesgue measure $0$. It immediately follows that the singular set is discrete in the sense that it has no interior points. 

To describe our method, we let $\tau=\frac{1}{i},\ i=1,2,\cdots$.
 We approximate $E(z)$ by
\begin{equation}\label{pht}
	E_\tau( z )=\frac{1}{p}(| z |^2+\tau)^{\frac{p}{2}}+\beta_0(| z |^2+\tau)^{\frac{1}{2}}.
\end{equation}
Then formulate our approximating problems as follows:
\begin{eqnarray}
	-\dest+\tau\ln\rho +a u &=&f\ \  \mbox{in $\Omega$},\label{ot1}\\
	-\mdiv\left[\nabla_z E_\tau(\nabla u)\right] +\tau u &=&\ln\rho \ \ \ \mbox{in $\Omega$,}\label{ot4}\\
	\nabla u\cdot\nu=\nabla\rho\cdot\nu&=&0\ \ \ \mbox{on $\partial\Omega$}.\label{ot2}
\end{eqnarray}
Here we have transformed the exponential nonlinearity into the logarithmic one. This is essentially the same as 
 we did in \cite{X1}.
In the limit ($\tau\ra 0$), \eqref{ot4} becomes
$$
	\rho\in e^{-\mdiv\left(\partial_z E(\nabla u)\right) }.
$$
Substitute this into \eqref{ot1} in the limit to obtain the original equation \eqref{sta1}.

 Our starting point is the following three a priori estimates (see Claims \ref{c43} and \ref{c44} below)
\begin{eqnarray}
	\left|\io \ln \rho dx\right|&\leq &c,\label{rlog}\\
	\|u\|_{W^{1,p}(\Omega)}&\leq & c,\nonumber\\
	\io|\nabla\sqrt{\rho }|^2dx&\leq &c.\nonumber
\end{eqnarray}
Here and in what follows the letter $c$ denotes a generic positive constant. In theory, its value can be computed from various given data. We must extract enough information from these three estimates and the three equations \eqref{ot1}-\eqref{ot2} to justify passing to the limit. The first issue is that to be able to apply Poincar\'{e}'s inequality (see Lemma \ref{poin} below) we need to know the average of $\rho $ over a set of positive measure is finite and \eqref{rlog} is far from doing that. We must bridge this gap to prevent the $\rho$-component of our approximate solutions from converging to infinity a.e. on $\Omega$ \cite{X1}. The second issue is how to estimate the function $\ln \rho $. 
To be specific, for each $\varepsilon>0$ we have $$\ln\rho\leq \rho ^\varepsilon\ \ \mbox{for $\rho$ sufficiently large}.$$
That is, we can easily gain the integrability of $\ln\rho$ from $\rho$ on the set where $\rho$ is large. However, for the same idea to work near $\rho=0$ we must have some estimates for $\rho ^{-\varepsilon}$ there.
Note that $\rho$ satisfies a non-homogeneous equation, which, by itself, does not prevent its solutions from having large zero sets or vanishing of any order at a zero point \cite{GL}. This constitutes the main difficulty of our analysis.

In the time-dependent case where
 the exponent is linear, i.e., $p=2, \beta_0=0$, it is possible to show that the singular set is empty by requiring the given data are suitably small \cite{GM,LS,PX}. In fact, one can obtain even higher regularity properties of solutions for the case \cite{A}.


The subdifferential of the second term in $E(z)$ gives rise to the 1-Laplace operator, which has the formal expression
$$\mbox{div}\left(\frac{\nabla u}{|\nabla u|}\right).$$ To give a mathematical interpretation of the term $\frac{\nabla u}{|\nabla u|}$, we follow the tradition in the monotone operator theory \cite{H}. Also see \cite{KV}. That is, we say
\begin{equation}\label{pdef}
\varphi(x)=\frac{\nabla u(x)}{|\nabla u(x)|}\Leftrightarrow \varphi(x)\in \partial_zH(\nabla u(x))\ \ \mbox{for a.e. $x\in \Omega$,}
\end{equation} where 
\begin{equation}
H(z)=|z|.\nonumber
\end{equation}
With this in mind, we can give the following definition of a weak solution.
\begin{defn}\label{df1}
	We say that a triplet $(u,\rho, \varphi)$ is a weak solution to \eqref{sta1}-\eqref{sta3} if the following conditions hold:
	\begin{enumerate}
			\item[(D1)]
		$\rho\in W^{2,p^*}(\Omega)$ with $\rho\geq 0$,
		$u\in W^{1,p}(\Omega)$, $\varphi\in (L^\infty(\Omega))^N$,
		$\mdiv(|\nabla u|^{p-2}\nabla u+\beta_0\varphi)\in \mathcal{M}(\overline{\Omega})\cap \left(W^{1,p}(\Omega)\right)^*$, where $p^*=\frac{Np}{N-p}$, $\left(W^{1,p}(\Omega)\right)^*$ is the dual space of $W^{1,p}(\Omega)$,  and $\mathcal{M} (\overline{\Omega})$ is the space of bounded Radon measures on $\overline{\Omega}$;
		\item[(D2)] For each $\ve\in (0,2)$  set
		\begin{equation}\label{la1}
			\Omega_\varepsilon=\{x_0\in \Omega:\limsup_{R\rightarrow 0}\frac{1}{R^{N+2-\ve}}\int_{\br}\rho(y)dy>0\},
		\end{equation}
	where  $\br$ is the ball centered at $x_0$ with radius $R$. Then every point in $\Omega_0\equiv\cup_{0<\ve<2}\Omega_\varepsilon$
	is regular and $\left|\Omega\setminus\Omega_0\right|=0$. Consequently, $\Omega_0$ is open and $\nnu,\ \ln\rho\in L^\infty_{\textup{loc}}(\Omega_0)$.  In addition, we have
		\begin{eqnarray}
			- \Delta \rho+au &=&f \ \  \mbox{a.e. on $\Omega$,}\label{rrrr2}\\
			-\mdiv(|\nabla u|^{p-2}\nabla u+\beta_0\varphi)&=&\ln\rho\ \  \mbox{a.e. on $\Omega$,}\label{rrrr1}\\
			\nabla \rho\cdot\nu &=& 0 \ \ \ \mbox{a.e. on $\po$}.\nonumber
		\end{eqnarray}
		The second boundary condition \eqref{sta3} is satisfied in the sense
		$$ \langle -\mdiv(|\nabla u|^{p-2}\nabla u+\beta_0\varphi), \xi \rangle
		=\io (|\nabla u|^{p-2}\nabla u+\beta_0\varphi)\cdot \nabla\xi \, dx \ \ \ \mbox{for all $\xi\in W^{1,p}(\Omega)$,}$$
		where $\langle \cdot,\cdot \rangle$ is the duality pairing between $W^{1,p}(\Omega)$ 
		and $\left(W^{1,p}(\Omega)\right)^*$
		\item[(D3)] $\varphi(x)=\frac{\nabla u(x)}{|\nabla u(x)|}$ in the sense of \eqref{pdef}.
	\end{enumerate}
\end{defn}

Our main result is the following 
\begin{thm}[Main theorem]\label{th1.1}	Assume that (H1) -(H2) hold and $\Omega$ is a bounded domain in $\mathbb{R}^N$ whose boundary $\po$ is either convex or $C^{2}$.
	Then there is a 
	weak solution to \eqref{sta1}-\eqref{sta3} in the sense of Definition \ref{df1}. 
\end{thm}

If $x_0$  is a singular point, we see from \eqref{la1} that
\begin{equation}
	\int_{B_{R}(x_0)}\rho dx=o(R^{N+2-\ve})\ \ \mbox{for each $\ve\in (0,2)$.}\nonumber
\end{equation}
That is, $\rho$ vanishes of at least $N+2-\ve$ order at a singular point $x_0$ \cite{GL}. As we indicated earlier, it does not seem to be possible to estimate the Hausdorff dimension of the singular set as was done in \cite{G}. 

We have included the 1-Laplacian in our analysis. The mathematical properties of this operator are still not well-understood \cite{GG,GK}. We must make sure that certain properties of the p-Poisson equation remain valid when the 1-Laplacian is present. It turns out that the p-Laplacian can always dominate the 1-Laplacian in the situation under our consideration. See Lemma \ref{l21} and \eqref{nub} below. More significantly, we have obtained a sufficient condition \eqref{la1} for a point to be regular. The condition (D2) implies that if $x_0\in\Omega$ is a Lebesgue point of $\rho$ and $\rho(x_0)>0$ then the exponential term behaves well in a neighborhood of $x_0$. By \eqref{rrrr1}, the singular part of the exponent is contained in $\Omega\setminus\Omega_0$.  As observed in \cite{X1}, we formally have
$$\Omega\setminus\Omega_0=\left\{-\mdiv(|\nabla u|^{p-2}\nabla u+\beta_0\varphi)=-\infty\right\}.$$
Thus the natural exponential function can cancel out the singular part so that \eqref{rrrr2} still holds.


This paper is organized as follows:
In Section 2 we collect a few known results. Three key preparatory lemmas are established in Section 3.  The proof of the main theorem is given in Section 4.

\section{Preliminaries}
In this section we collect a few known results that are useful to us. We begin with some elementary inequalities.

If $a, b\in [0,\infty)$ and $\beta>0$, we have
$$
(a+b)^\beta\leq \left\{\begin{array}{ll}
	a^\beta+b^\beta&\mbox{if $\beta\leq 1$,}\\
	2^{\beta-1}\left(	a^\beta+b^\beta\right)&\mbox{if $\beta>1$}.
\end{array}\right.
$$
That is, we always have $(a+b)^\beta\leq c\left(	a^\beta+b^\beta\right)$. When an occasion arises for this inequality, it will be used without acknowledgment. Other frequently used inequalities include Young's inequality
\begin{equation}\label{young}
	ab\leq \varepsilon a^p+\frac{1}{\varepsilon^{q/p}}b^q, \ \
	\mbox{ $\varepsilon>0,\, p, \, q>1$ with $\frac{1}{p}+\frac{1}{q}=1$}
\end{equation}
and the interpolation inequality
\begin{equation}\label{linterp}
	\|f\|_q\leq\varepsilon\|f\|_r+\varepsilon^{-\sigma} \|f\|_p,\ \ \mbox{
		$\varepsilon>0,\ p\leq q\leq r$, and $\sigma=\left(\frac{1}{p}-\frac{1}{q}\right)/\left(\frac{1}{q}-\frac{1}{r}\right)$},
\end{equation}
where 	$\|\cdot\|_p$ denotes the norm in the space $L^p(\Omega)$.
In the applications of the Sobolev inequality 
\begin{equation}
	\|u\|_{p^*}\leq c(\|\nabla u\|_p+\|u\|_1),\ \ p^*=\frac{Np}{N-p},\nonumber
\end{equation} it is understood that $1\leq p<N$ because the case where $p\geq N$ is much simpler in our context.

Our existence theorem is based upon the following fixed point theorem, which is often called the Leray-Schauder Theorem (\cite{GT}, p.280).
\begin{lem}\label{lsf}
	Let $B$ be a map from a Banach space $\mathcal{B}$ into itself. Assume:
	\begin{enumerate}
		\item[(LS1)] $B$ is continuous;
		\item[(LS2)] the images of bounded sets of $B$ are precompact;
		\item[(LS3)] there exists a constant $c$ such that
		$$\|z\|_{\mathcal{B}}\leq c$$
		for all $z\in\mathcal{B}$ and $\sigma\in[0,1]$ satisfying $z=\sigma B(z)$.
	\end{enumerate}
	Then $B$ has a fixed point.
\end{lem}
\begin{lem}\label{poin}
	Let $\Omega$ be a bounded domain in $\RN$ with Lipschitz boundary and $1\leq p<N$.
	Then there is a positive number $c=c(N)$ such that
	\begin{equation}
	\|u-u_S\|_{p^*}\leq \frac{cd^{N+1-\frac{p}{N}}}{|S|^{\frac{1}{p}}}\|\nabla u\|_p \ \ \mbox{for each $u\in W^{1,p}(\Omega)$,}\nonumber
	\end{equation}
	where $S$ is any measurable subset of $\Omega$ with $|S|>0$, $u_S=\frac{1}{|S|}\int_S udx$, and $d$ is the diameter of $\Omega$.
\end{lem}	
This lemma can be inferred from Lemma 7.16 in \cite{GT}. Also see \cite{G,PS}. It is a version of Poincar\'{e}'s inequality. 

 \begin{lem}\label{ynb}
	Let $\{y_n\}, n=0,1,2,\cdots$, be a sequence of positive numbers satisfying the recursive inequalities
	\begin{equation}\label{ynb1}
		y_{n+1}\leq cb^ny_n^{1+\alpha}\ \ \mbox{for some $b>1, c, \alpha\in (0,\infty)$.}
	\end{equation}
	If
	\begin{equation*}
		y_0\leq c^{-\frac{1}{\alpha}}b^{-\frac{1}{\alpha^2}},
	\end{equation*}
	then $\lim_{n\rightarrow\infty}y_n=0$.
\end{lem}
This lemma can be found in (\cite{D}, p.12).

\section{Three key lemmas}
In this section we prove three key lemmas. They lay the foundation for our existence theorem.

The first lemma deals with the exponent in our problem. 	

Let $E_\tau$ be given as in \eqref{pht}.
Define
  \begin{equation}\label{ftd}
  	F_\tau(s)=(s+\tau)^{\frac{p-2}{2}}+\beta_0(s+\tau)^{-\frac{1}{2}}\ \ \mbox{on $[0,\infty)$}.
  \end{equation}
Then we can easily verify
\begin{equation}
	\nabla_z E_\tau( z )= F_\tau(|z|^2)z.\nonumber
\end{equation}
  In the subsequent proof we will assume
  \begin{equation}\label{conpn}
  	1<p\leq 2,\ \ N>2.
  \end{equation}
  This is done mainly for the convenience in applying the Sobolev inequality and also avoiding non-essential complications. Cases where $p>2$ and/or $N=2$ \cite{X2} are simpler, and we leave them to the interested reader.
For $p\in (1, 2]$, we have  (see, e.g., \cite{X2}) that 
  \begin{eqnarray}
  	\left((|z|^2+\tau)^{\frac{p-2}{2}}z-	(|y|^2+\tau)^{\frac{p-2}{2}}y\right)\cdot(z-y)&\geq&(p-1)\left(1+|y|^2+|z|^2\right)^{\frac{p-2}{2}}|z-y|^2,\nonumber\\
  	\left((|z|^2+\tau)^{-\frac{1}{2}}z-	(|y|^2+\tau)^{-\frac{1}{2}}y\right)\cdot(z-y)&\geq &0,\ \ z ,y\in\mathbb{R}^N.\nonumber
  \end{eqnarray}
  Subsequently,
\begin{eqnarray}
	\lefteqn{(F_\tau(|z|^2)z-F_\tau(|y|^2)y)\cdot( z -y)}\nonumber\\
	&\geq& (p-1)\left(1+|y|^2+|z|^2\right)^{\frac{p-2}{2}}|z-y|^2\ \ \ \mbox{for all $ z ,y\in\mathbb{R}^N$.}\label{uni}
\end{eqnarray}

\begin{lem}\label{l21}
	Let $\Omega$ be a bounded domain in $\mathbb{R}^N$ with Lipschitz boundary $\po$. Consider the problem
	\begin{eqnarray}
		-\mdiv\left[\nabla_zE_\tau(\nabla u)\right] +\tau u&=& f\ \ \textup{in $\Omega$,}\label{of1}\\
		\nabla u\cdot\nu&=&0\ \ \textup{on $\po$,}\label{of2}
	\end{eqnarray}
	where  $ p>1$ and $f\in L^{\frac{Np}{Np-N+p}}(\Omega)$. Then there is a unique weak solution $u$ to the  problem in the space $W^{1,p}(\Omega)$. 
	Furthermore,
	if $f$ also lies in the space $L^{s}(\Omega)$ with
	\begin{equation}\label{of3}
		s>\frac{N}{p},
	\end{equation}  we have, for any ball $\br\subset\Omega$,
\begin{equation}\label{ulb}
\sup_{\brh}u^\pm\leq c\left(\avint_{B_{R}(x_0)}(u^\pm)^pdx\right)^{\frac{1}{p}}+R\left(1+\tau^{\frac{1}{2}}\right)+\left(R^{\frac{ps-N}{ s}}\|f^\pm\|_ {s, B_R(x_0)}\right)^{\frac{1}{p-1}},
\end{equation}
where $c=c(N,s, p)$ and $\avint_{B_{R}(x_0)}(u^\pm)^pdx=\frac{1}{|\br|}\int_{B_{R}(x_0)}(u^\pm)^pdx$. 
\end{lem}

Of course, this lemma is well known if $\beta_0=0$. For our purpose, we need precise information on how $u$ is bounded by $f$. We will offer a proof.
\begin{proof} 
	
	The existence of a solution can be established by showing the functional
	$$\io E_\tau(\nabla u)dx-\io f udx$$ has a minimizer in $W^{1,p}(\Omega)$, while the uniqueness can be inferred from \eqref{uni}. We shall omit the details.
	
The proof of \eqref{ulb} is based upon the De Giorgi iteration scheme. Let $x_0, R$ be given as in the lemma. Define a sequence of concentric balls $B_{R_n}(x_0)$ in $\Omega$ as follows:
	\begin{equation}
		B_{R_n}(x_0)=\{x:|x-x_0|<R_n\},\nonumber
	\end{equation}
	where
	\begin{equation}
		R_n=\frac{R}{2}+\frac{R}{2^{n+1}},\ \ \ n=0,1,2,\cdots.\nonumber
	\end{equation} Choose a sequence of smooth functions $\theta_n$ so that
	\begin{eqnarray}
		\theta_{n+1}(x)&=& 1 \ \ \mbox{in $B_{R_{n+1}}(x_0)$},\nonumber\\
		\theta_{n+1}(x)&=&0\ \ \mbox{outside $B_{R_{n}}(x_0)$},\nonumber\\
		|\nabla \theta_{n+1}(x)|&\leq & \frac{c2^n}{R}\ \ \mbox{for each $x\in \mathbb{R}^N$,}\ \ \ \mbox{and}\nonumber\\
		0&\leq &\theta_{n+1}(x)\leq 1\ \ \mbox{in $\mathbb{R}^N$.}\nonumber
	\end{eqnarray}
	Select 
	$$
	k> 0.
	$$
	as below. Let
	\begin{eqnarray}
		k_n&=&k-\frac{k}{2^{n+1}},\nonumber\\
		Q_n&=&\{x\in\brn: u(x)\geq k_{n+1}\},\label{qn}\\
		y_{n}&=&\int_{\brn}\left[(u-k_{n})^+\right]^{p}dx, \ \ n=0,1,2,\cdots.\nonumber
	\end{eqnarray}Using $(u-k_{n+1})^+\theta_{n+1}^2$ as a test function in \eqref{of1} yields
	\begin{eqnarray}
		\lefteqn{	\io\left(|\nabla(u-k_{n+1})^+|^2+\tau\right)^{\frac{p}{2}-1}|\nabla(u-k_{n+1})^+|^2\theta_{n+1}^2dx}\nonumber\\
		&\leq&-\io(u-k_{n+1})^+\theta_{n+1}\nabla\theta_{n+1}\cdot F_\tau(\nnu^2)\nabla u dx+\io f^+(u-k_{n+1})^+\theta_{n+1}^2dx\nonumber\\
		&\equiv&I_1+I_2.\nonumber
	\end{eqnarray}
	From \eqref{conpn} and \eqref{ftd}, we have
	$$\frac{p}{p-1}\geq 2,\ \ |F_\tau(\nnu^2)\nabla u|\leq \nnu^{p-1}+\beta_0.$$
	With these in mind, we derive
	\begin{eqnarray}
		I_1&\leq &\frac{c2^n}{R}\io\theta_{n+1}(u-k_{n+1})^+\left(\nnu^{p-1}+\beta_0\right)dx\nonumber\\
		&\leq &\varepsilon\int_{Q_n}\theta_{n+1}^2|\nabla(u-k_{n+1})^+|^pdx+\frac{c(\ve)2^{pn}}{R^p}y_{n}+\frac{c2^n}{R}y_n^{\frac{1}{p}}|Q_n|^{1-\frac{1}{p}},\ \ \ve>0.\nonumber
	\end{eqnarray}
	As for $I_2$, we deduce from  the Sobolev inequality that
	\begin{eqnarray}
		I_2&\leq&\left(\int_{Q_n}|f^+|^{\frac{Np}{Np-N+p}}dx\right)^{\frac{Np-N+p}{Np}}\left(\io\left[(u-k_{n+1})^+\theta_{n+1}^2\right]^{\frac{Np}{N-p}}dx\right)^{\frac{N-p}{Np}}\nonumber\\
		&\leq &c\|f^+\|_ {s,\br}|Q_n|^{\frac{Np-N+p}{Np}-\frac{1}{ s}}\left(\left(\io\theta_{n+1}^2|\nabla(u-k_{n+1})^+|^pdx\right)^{\frac{1}{p}}+\frac{c2^n}{R}y_{n}^{\frac{1}{p}}\right)\nonumber\\
		&\leq &\varepsilon\int_{Q_n}\theta_{n+1}^2|\nabla(u-k_{n+1})^+|^pdx+\frac{c2^{pn}}{R^p}y_{n}+c(\ve)\left(\|f^+\|_ {s, B_R(x_0)}\right)^{\frac{p}{p-1}}|Q_n|^{\frac{Np-N+p}{N(p-1)}-\frac{p}{ s(p-1)}}.\nonumber
	\end{eqnarray}
	On the other hand, we can deduce from \eqref{conpn} that
	\begin{eqnarray}
		\lefteqn{\int_{Q_n}\theta_{n+1}^2|\nabla(u-k_{n+1})^+|^pdx}\nonumber\\
		&\leq&\int_{Q_n}\theta_{n+1}^2\left(|\nabla(u-k_{n+1})^+|^2+\tau\right)^{\frac{p}{2}-1}\left(|\nabla(u-k_{n+1})^+|^2+\tau\right)dx\nonumber\\
		&\leq&2\varepsilon\int_{Q_n}\theta_{n+1}^2|\nabla(u-k_{n+1})^+|^pdx+\frac{c2^{pn}}{R^p}y_{n}+\frac{c2^n}{R}y_n^{\frac{1}{p}}|Q_n|^{1-\frac{1}{p}}\nonumber\\
		&&+c(\ve)\left(\|f^+\|_ {s, B_R(x_0)}\right)^{\frac{p}{p-1}}|Q_n|^{\frac{Np-N+p}{N(p-1)}-\frac{p}{ s(p-1)}}+\tau^{\frac{p}{2}}|Q_n|.\nonumber
	\end{eqnarray}
	By taking $\ve$ suitably small, we arrive at
	\begin{eqnarray}
		\io\theta_{n+1}^2|\nabla(u-k_{n+1})^+|^pdx	&\leq &\frac{c2^{pn}}{R^p}y_{n}+\frac{c2^n}{R}y_n^{\frac{1}{p}}|Q_n|^{1-\frac{1}{p}}\nonumber\\
		&&+c(\ve)\left(\|f^+\|_ {s, B_R(x_0)}\right)^{\frac{p}{p-1}}|Q_n|^{\frac{Np-N+p}{N(p-1)}-\frac{p}{ s(p-1)}}+\tau^{\frac{p}{2}}|Q_n|.\nonumber
	\end{eqnarray}
	Apply the Sobolev inequality again to deduce
	\begin{eqnarray}
		y_{n+1}&\leq&\io\left[\theta_{n+1}^2(u-k_{n+1})^+\right]^pdx\nonumber\\
		&\leq&\left(\io\left[\theta_{n+1}^2(u-k_{n+1})^+\right]^{\frac{Np}{N-p}}dx\right)^{\frac{N-p}{N}}|Q_n|^{\frac{p}{N}}\nonumber\\
		&\leq&c\left(\io\theta_{n+1}^2|\nabla(u-k_{n+1})^+|^pdx+\frac{2^{pn}}{R^p}y_{n}\right)|Q_n|^{\frac{p}{N}}\nonumber\\
		&\leq &\frac{c2^{pn}}{R^p}y_{n}|Q_n|^{\frac{p}{N}}+\frac{c2^n}{R}y_n^{\frac{1}{p}}|Q_n|^{1-\frac{1}{p}+\frac{p}{N}}\nonumber\\
		&&+c\left(\|f^+\|_ {s, B_R(x_0)}\right)^{\frac{p}{p-1}}|Q_n|^{1+\frac{(ps-N)p}{(p-1)N s}}+\tau^{\frac{p}{2}}|Q_n|^{1+\frac{p}{N}}\nonumber\\
		&\equiv&J_1+J_2+J_3+J_4.\nonumber
	\end{eqnarray}
	Note that
	\begin{eqnarray}
		y_n&\geq &\int_{Q_n}\left[(u-k_{n})^+\right]^pdx\geq \frac{k^p}{2^{pn+2p}}\left|Q_n\right|.\nonumber
	\end{eqnarray}
	Also remember from  \eqref{of3} that $\frac{(ps-N)p}{(p-1)N s}>0$. Thus we can pick a number $\alpha$ with the property
	$$
	0<\alpha\leq \min\left\{\frac{(ps-N)p}{(p-1)N s},\frac{p}{N}\right\}.
	$$
	Then $\alpha>0$ due to \eqref{of3}.
	Subsequently,
	\begin{eqnarray}
		J_1&=&\frac{c2^{pn}}{R^p}y_{n}|Q_n|^{\frac{p}{N}-\alpha+\alpha}\leq\frac{c2^{pn}}{R^{N\alpha}}y_{n}|Q_n|^{\alpha}\leq \frac{c2^{(\alpha+1)pn}}{R^{N\alpha}k^{p\alpha}}y_{n}^{1+\alpha},\nonumber\\
		J_2&=&\frac{c2^n}{R}|Q_n|^{\frac{p}{N}-\alpha}y_n^{\frac{1}{p}}|Q_n|^{1-\frac{1}{p}+\alpha}\leq\frac{c2^nR^{p-1}}{R^{N\alpha}}y_n^{\frac{1}{p}}|Q_n|^{1-\frac{1}{p}+\alpha}\leq \frac{c2^{(\alpha+1)pn}R^{p-1}}{R^{N\alpha}k^{p(1-\frac{1}{p}+\alpha)}}y_n^{1+\alpha}\nonumber\\
		&\leq& \frac{c2^{(\alpha+1)pn}}{R^{N\alpha}k^{p\alpha}}y_{n}^{1+\alpha}\ \ \mbox{if $R\leq k$},\nonumber\\
		J_3&\leq&\frac{c\left(\|f^+\|_ {s, B_R(x_0)}\right)^{\frac{p}{p-1}}R^{\frac{(ps-N)p}{(p-1) s}}}{R^{N\alpha}}|Q_n|^{1+\alpha}\nonumber\\
		&\leq &\frac{c\left(R^{\frac{ps-N}{ s}}\|f^+\|_ {s, B_R(x_0)}\right)^{\frac{p}{p-1}}2^{(\alpha+1)pn}}{R^{N\alpha}k^{p(1+\alpha)}}y_{n}^{1+\alpha}\nonumber\\
		&\leq&\frac{c2^{(\alpha+1)pn}}{R^{N\alpha}k^{p\alpha}}y_{n}^{1+\alpha}\ \ \mbox{if $\left(R^{\frac{ps-N}{ s}}\|f^+\|_ {s, B_R(x_0)}\right)^{\frac{1}{p-1}}\leq k$,}\nonumber\\
		J_4&\leq &\frac{cR^p\tau^{\frac{p}{2}}}{R^{N\alpha}}|Q_n|^{1+\alpha}\leq \frac{cR^p\tau^{\frac{p}{2}}2^{(\alpha+1)pn}}{R^{N\alpha}k^{p(1+\alpha)}}y_{n}^{1+\alpha}\nonumber\\
		&\leq &\frac{c2^{(\alpha+1)pn}}{R^{N\alpha}k^{p\alpha}}y_{n}^{1+\alpha}\ \ \mbox{if $R\tau^{\frac{1}{2}}\leq k$.}\nonumber
	\end{eqnarray}
	We choose $k$ so large that 
	\begin{equation}\label{aklb}
		k\geq R\left(1+\tau^{\frac{1}{2}}\right)+\left(R^{\frac{ps-N}{ s}}\|f^+\|_ {s, B_R(x_0)}\right)^{\frac{1}{p-1}}.
	\end{equation}
	Then we have
	$$
	y_{n+1}\leq \frac{c2^{(\alpha+1)pn}}{R^{N\alpha}k^{p\alpha}}	y_n^{1+\alpha}.
	$$
	Now we are in a position to apply Lemma \ref{ynb}.  To do this, we further require  $k$ to satisfy
	$$
	y_0=\int_{B_{R}(x_0)}(u^+)^pdx\leq \left(\frac{c}{R^{N\alpha}k^{p\alpha}}\right)^{-\frac{1}{\alpha}}2^{-\frac{(\alpha+1)p}{\alpha^2}}=cR^Nk^p.
	$$
	Then
	$$
	\max_{B_{\frac{R}{2}}(x_0)}u\leq k\ \ \mbox{on $\Omega$.}
	$$
	In view of \eqref{aklb}, it is enough for us to take 
	$$
	k=c\left(\avint_{B_{R}(x_0)}(u^+)^pdx\right)^{\frac{1}{p}}+R\left(1+\tau^{\frac{1}{2}}\right)+\left(R^{\frac{ps-N}{ s}}\|f^+\|_ {s, B_R(x_0)}\right)^{\frac{1}{p-1}}.
	$$
	This implies the desired result.
\end{proof}

In the situation considered in Lemma \ref{l21} $u$ is also bounded in the whole domain and we have the estimate
\begin{equation}\label{of4}
	\|u^\pm\|_\infty \leq c\|u^\pm\|_1+\left(\|f^\pm\|_s\right)^{\frac{1}{p-1}}+\sqrt{\tau},
\end{equation}
where $c$ depends only on $N, p, s, \Omega$.
The proof is simpler, and we shall omit the details.

\begin{lem}
Let the hypotheses and notion of Lemma \ref{l21} hold. If, in addition, $f\in L^\ell(\Omega)$ with $\ell>N$, we have, for any ball $\br\subset\Omega$,
\begin{equation}\label{nub}
\sup_{\brh}|\nabla u|\leq c\avint_{B_{R}(x_0)}|\nabla u|dx+c\sqrt{\tau}
+\left(R^{1-\frac{N}{\ell}}\||f|+\tau|u|\| _{\ell, B_R(x_0)}\right)^{\frac{1}{p-1}}+c,	
\end{equation}
 where $c=c(N,\ell, p)$.
\end{lem}

This lemma is more or less known. See, e.g., \cite{T2}. We will offer a simpler proof here.
\begin{proof}
	To obtain \eqref{nub}, we first derive a differential inequality satisfied by
	$$
	w=|\nabla u|^2.
	$$
	To this end, we first observe that our solution $u$ actually lies in $W^{2,\ell}(\Omega)$ \cite{LU,X2}. Thus we can	differentiate \eqref{of1} with respect to $x_j, j\in \{1, \cdots, N\}$, to derive
	$$
	-\mbox{div}\left(\nabla_z^2 E_\tau(\nabla u)\nabla u_{x_j}\right)+\tau u_{x_j} =f_{x_j}.
	$$
	It is easy to verify
	\begin{equation}\label{3}
		\nabla^2E_\tau(z)=(|z|^2+\tau)^{\frac{p-2}{2}}\left(I+(p-2)\frac{z\otimes z}{|z|^2+\tau}\right)+\beta_0(|z|^2+\tau)^{-\frac{1}{2}}\left(I-\frac{z\otimes z}{|z|^2+\tau}\right)
		.
	\end{equation}
Consequently,
\begin{eqnarray}
	\nabla^2E_\tau(\nabla u) z \cdot z &=&(w+\tau)^{\frac{p-2}{2}}\left(| z |^2+(p-2)\frac{(\nabla u\cdot z )^2}{w+\tau}\right)\nonumber\\
	&&+\beta_0(w+\tau)^{-\frac{1}{2}}\left(| z |^2-\frac{(\nabla u\cdot z )^2}{w+\tau}\right) \nonumber\\
		&\geq &(w+\tau)^{\frac{p-2}{2}}\left(| z |^2-(2-p)^+\frac{w| z |^2}{w+\tau}\right)+\beta_0(w+\tau)^{-\frac{1}{2}}\left(| z |^2-\frac{w| z |^2}{w+\tau}\right)\nonumber\\
		&\geq&	(1-(2-p)^+)(w+\tau)^{\frac{p-2}{2}}| z |^2.\label{co}
	\end{eqnarray}
	We also have
	\begin{eqnarray}
		|\nabla^2E_\tau(\nabla u)|\leq 	c\left[(w+\tau)^{\frac{p-2}{2}}+\beta_0(w+\tau)^{-\frac{1}{2}}\right].\label{co1}
	\end{eqnarray}
	Multiply through \eqref{3} by $u_{x_j}$ to obtain
	\begin{eqnarray}
		-\frac{1}{2}\mbox{div}\left(\partial_z^2 E_\tau(\nabla u)\nabla u_{x_j}^2\right)+\partial_z^2 E_\tau(\nabla u)\nabla u_{x_j}\cdot \nabla u_{x_j}
		+\tau u_{x_j}^2
		=f_{x_j}u_{x_j}.\label{4}
	\end{eqnarray}
	By \eqref{co},
	\begin{equation}
		\nabla_z^2 E_\tau(\nabla u)\nabla u_{x_j}\cdot \nabla u_{x_j}\geq 0.\nonumber
	\end{equation}
	Use this in \eqref{4}, sum up the resulting inequality over $j$, and thereby obtain
	\begin{equation}\label{w1}
		-\mbox{div}\left(\nabla_z^2 E_\tau(\nabla u)\nabla w\right)\leq 2\nabla f\nabla u.
	\end{equation}

Let $R, x_0$ be given as in the lemma. Fix $\sigma\in(0,1)$. Let 
\begin{equation}
	R_n=(1-\sigma)R+\frac{\sigma R}{2^{n}},\ \ \ n=0,1,2,\cdots.\nonumber
\end{equation}
Define $\theta_n$ as before. The only difference is that we now have
\begin{eqnarray}
	|\nabla \theta_{n+1}(x)|&\leq & \frac{c2^n}{\sigma R}\ \ \mbox{for each $x\in \mathbb{R}^N$.}\nonumber
\end{eqnarray}
	Select
	\begin{equation}\nonumber
		K>2
	\end{equation}
	as below.
	Set
	\begin{eqnarray}
			v&=&(w+\tau)^{\frac{p}{2}}.\nonumber
	\end{eqnarray}
Then define
	\begin{eqnarray}
		K_n&=&K-\frac{K}{2^{n+1}},\ \ Q_n=\{x\in\brn: v\geq K_{n+1}\}\ \ \ n=0,1,2,\cdots\label{qn}
		\end{eqnarray}
Obviously,
\begin{equation}\label{conK2}
	K_n\geq 1\ \ \mbox{for all $n$.}
\end{equation} We continue to assume \eqref{conpn}. Use $\theta_{n+1}^2( v-K_{n+1})^+$ as a test function in \eqref{w1} to derive
	\begin{eqnarray}
		\lefteqn{\io\theta_{n+1}^2\nabla^2E_\tau(\nabla u)\nabla w\cdot\nabla( v-K_{n+1})^+dx}\nonumber\\
		&\leq&-2\io\theta_{n+1}\nabla\theta_{n+1}\cdot \nabla^2E_\tau(\nabla u)\nabla w( v-K_{n+1})^+dx\nonumber\\
		&&+2\io \nabla f\cdot \nabla u\theta_{n+1}^2( v-K_{n+1})^+dx.\label{th1}
	\end{eqnarray}
	Now we proceed to analyze each term in the above inequality. In view of \eqref{co}, we have 
	\begin{eqnarray}
		\lefteqn{	\io\theta_{n+1}^2\nabla^2E_\tau(\nabla u)\nabla w\cdot\nabla( v-K_{n+1})^+dx}\nonumber\\
		&=&\frac{2}{p}\io\theta_{n+1}^2(w+\tau)^{1-\frac{p}{2}}\nabla^2E_\tau(\nabla u)\nabla ( v-K_{n+1})^+\cdot\nabla( v-K_{n+1})^+dx\nonumber\\
		&\geq&\frac{2(p-1)}{p}\io\theta_{n+1}^2|\nabla( v-K_{n+1})^+|^2dx.\label{th2}
	\end{eqnarray}
Note from  \eqref{conK2} and \eqref{qn} that
$$(w+\tau)^{-\frac{p-1}{2}}=v^{-\frac{p-1}{p}}\leq 1 \ \ \mbox{in $Q_n$.}$$
This together with \eqref{co1} enables us to estimate the second term in \eqref{th1} as follows:
	\begin{eqnarray}
		\lefteqn{-2\io\theta_{n+1}\nabla\theta_{n+1}\cdot \nabla^2E_\tau(\nabla u)\nabla w( v-K_{n+1})^+dx}\nonumber\\
		&\leq &\frac{c2^n}{\sigma R}\io\theta_{n+1}\left[1+\beta_0(w+\tau)^{-\frac{p-1}{2}}\right]|\nabla( v-K_{n+1})^+|( v-K_{n+1})^+dx\nonumber\\
		&\leq &\varepsilon\io\theta_{n+1}^2|\nabla( v-K_{n+1})^+|^2dx+\frac{c(\varepsilon)4^n}{(\sigma R)^2}\int_{\brn}\left[( v-K_{n+1})^+\right]^2dx.\nonumber
	\end{eqnarray}
	As for the last integral in \eqref{th1}, we recall from
	\eqref{of1} that
	$$
	\mdiv\left(F_\tau(w)\nabla u\right)=\tau u-f.
	$$
	Consequently,
	\begin{eqnarray}
		\mbox{the last term in \eqref{th1}}&=&2\io \nabla f\cdot \nabla u\theta_{n+1}^2( v-K_{n+1})^+dx\nonumber\\
		&=&2\io \nabla f\cdot F_\tau(w)\nabla u\theta_{n+1}^2\frac{( v-K_{n+1})^+}{F_\tau(w)}dx\nonumber\\
		&=&-2\io f(\tau u-f)\theta_{n+1}^2\frac{( v-K_{n+1})^+}{F_\tau(w)}dx\nonumber\\
		&&-4\io  f \nabla u\cdot\theta_{n+1}\nabla\theta_{n+1}( v-K_{n+1})^+dx\nonumber\\
		&&-2\io  \theta_{n+1}^2f \nabla u\cdot F_\tau(w)\nabla \frac{( v-K_{n+1})^+}{F_\tau(w)}dx\nonumber\\
		&\equiv& I_1+I_2+I_3.\nonumber
	\end{eqnarray}
	We easily see from \eqref{ftd} that $$\frac{1}{F_\tau(w)}\leq (w+\tau)^{\frac{2-p}{2}}=v^{\frac{2-p}{p}}$$ Hence,
	\begin{eqnarray}
		I_1&\leq&c\io|(\tau u-f)f|\theta_{n+1}^2v^{\frac{2-p}{p}}( v-K_{n+1})^+dx\nonumber\\
		&\leq &c\int_{Q_n}(f^2+\tau^2u^2)v^{\frac{2}{p}}dx,\nonumber
	\end{eqnarray}
	where $Q_n$ is given as in \eqref{qn}.
	Similarly,
	\begin{eqnarray}
		I_2&\leq&\frac{c2^n}{\sigma R}\io  |f|w^{\frac{1}{2}}\theta_{n+1}( v-K_{n+1})^+dx\nonumber\\
		&\leq&\frac{c2^n}{\sigma R}\io  |f| v^{\frac{1}{p}}\theta_{n+1}( v-K_{n+1})^+dx\nonumber\\
		&\leq&\frac{c4^n}{(\sigma R)^2}\int_{Q_n}\left[( v-K_{n+1})^+\right]^2dx+c\int_{Q_n}f^2v^{\frac{2}{p}}dx.\nonumber
	\end{eqnarray}
	To estimate $I_3$, we observe that
	$$
	\frac{(v-K_{n+1})^+}{F_\tau(w)}=\left.\frac{s}{(s+K_{n+1})^{-\frac{2-p}{p}}+\beta_0(s+K_{n+1})^{-\frac{1}{p}}}\right|_{s=(v-K_{n+1})^+}
	$$
	Then we can easily check
	$$
	\left|\frac{d}{ds}\left(\frac{s}{(s+K_{n+1})^{-\frac{2-p}{p}}+\beta_0(s+K_{n+1})^{-\frac{1}{p}}}\right)\right|\leq \frac{c}{(s+K_{n+1})^{-\frac{2-p}{p}}+\beta_0(s+K_{n+1})^{-\frac{1}{p}}}.
	$$
	This immediately implies that
	$$
	\left|F_\tau(w)\nabla \frac{( v-K_{n+1})^+}{F_\tau(w)}\right|\leq c\left|\nabla(v-K_{n+1})^+\right|.
	$$
	Subsequently,
	\begin{eqnarray}
		I_3&=&-2\io  \theta_{n+1}^2f \nabla u\cdot F_\tau(w)\nabla \frac{( v-K_{n+1})^+}{F_\tau(w)}dx\nonumber\\
		&\leq &\varepsilon\io \theta_{n+1}^2|\nabla ( v-K_{n+1})^+|^2dx+c\int_{Q_n}f^2v^{\frac{2}{p}}dx.
		\label{th6}	
	\end{eqnarray}
	With the aid of \eqref{th2}-\eqref{th6}, we can deduce from \eqref{th1} that
	\begin{eqnarray}
		\lefteqn{\io\theta_{n+1}^2|\nabla( v-K_{n+1})^+|^2dx}\nonumber\\
		&\leq&\frac{c4^n}{(\sigma R)^2}\int_{Q_n}\left[( v-K_{n+1})^+\right]^2dx
		+c\int_{Q_n}(f^2+\tau^2u^2)v^{\frac{2}{p}}dx.\label{haa1}
	\end{eqnarray}
	Now set 
	$$
	y_n=\int_{B_{R_{n}}(x_0)}\left[( v-K_{n})^+\right]^{2}dx.
	$$
	We wish to show that the sequence $\{y_n\}$ satisfies \eqref{ynb1}. To this end, we estimate 
	\begin{eqnarray}
		y_n&\geq&\int_{Q_n}v^{2}\left(1-\frac{K_n}{K_{n+1}}\right)^{2}dx\geq\frac{1}{4^{n+2}}\int_{Q_n}v^{2}dx.\nonumber
	\end{eqnarray}
	Consequently,
	\begin{eqnarray}
		\int_{Q_n} (f^2+\tau^2u^2)v^{\frac{2}{p}}dx
		&\leq&\left(\int_{Q_n}v^{2}dx\right)^{\frac{1}{p}}\left(\int_{Q_n}(f^2+\tau^2u^2)^{\frac{p}{p-1}}dx\right)^{\frac{p-1}{p}}\nonumber\\
		&\leq&4^{\frac{2+n}{p}}y_n^{\frac{1}{p}}\||f|+\tau|u|\| _{\ell, B_R(x_0)}^2|Q_n|^{\frac{p-1}{p}-\frac{2}{\ell}}.\nonumber
		\end{eqnarray}
	Substituting this into \eqref{haa1} yields
	\begin{eqnarray}
		\io\theta_{n+1}^2|\nabla( v-K_{n+1})^+|^2dx\leq\frac{c4^n}{(\sigma R)^2}y_n+4^{\frac{2+n}{p}}y_n^{\frac{1}{p}}\||f|+\tau|u|\| _{\ell, B_R(x_0)}^2|Q_n|^{\frac{p-1}{p}-\frac{2}{\ell}}.\label{haa2}
	\end{eqnarray}
	By Poincar\'{e}'s inequality,
	\begin{eqnarray}
		y_{n+1}&\leq &\io \left[\theta_{n+1}( v-K_{n+1})^+\right]^{2}dx\nonumber\\
		&\leq&\left(\io \left[\theta_{n+1}( v-K_{n+1})^+\right]^{\frac{2N}{N-2}}dx\right)^{\frac{N-2}{N}}|Q_n|^{\frac{2}{N}}\nonumber\\
		&\leq &c\io \left|\nabla\left(\theta_{n+1}( v-K_{n+1})^+\right)\right|^2dx|Q_n|^{\frac{2}{N}}\nonumber\\
		&\leq&c\io \theta_{n+1}^2\left|\nabla( v-K_{n+1})^+\right|^2dx|Q_n|^{\frac{2}{N}}+\frac{c4^n}{(\sigma R)^2}y_n|Q_n|^{\frac{2}{N}}.\nonumber
	\end{eqnarray}
	This combined with \eqref{haa2} yields
	\begin{eqnarray}
		y_{n+1}&\leq&\frac{c4^n}{(\sigma R)^2}y_n|Q_n|^{\frac{2}{N}}+4^{\frac{2+n}{p}}y_n^{\frac{1}{p}}\||f|+\tau|u|\| _{\ell, B_R(x_0)}^2|Q_n|^{\frac{p-1}{p}-\frac{2}{\ell}+\frac{2}{N}}\nonumber\\
		&\equiv&N_1+N_2.\nonumber
	\end{eqnarray}
	Note that
	\begin{eqnarray}
		y_n&\geq& \int_{Q_n}(K_{n+1}-K_{n})^2dx\geq\frac{K^2}{4^{n+2}}|Q_n|.\nonumber
	\end{eqnarray}
	Recall that $\frac{2(\ell-N)}{N\ell}>0$. Thus we can pick a number $\alpha$ so that
	$$0<\alpha\leq \min\left\{\frac{2}{N}, \frac{2(\ell-N)}{N\ell}\right\}.$$
	It immediately follows that
	\begin{eqnarray}
		N_1&=&\frac{c4^n}{(\sigma R)^2}y_n|Q_n|^{\alpha+\frac{2}{N}-\alpha}\leq \frac{c4^{(1+\alpha)n}}{\sigma^2R^{N\alpha}K^{2\alpha}}y_n^{1+\alpha},\nonumber\\
	N_2	&\leq&\frac{4^{n(1+\alpha)}R^{\frac{2(\ell-N)}{\ell}}\||f|+\tau|u|\| _{\ell, B_R(x_0)}^2}{R^{N\alpha}K^{2\left(1-\frac{1}{p}+\alpha\right)}}y_n^{1+\alpha}\nonumber\\
	&\leq&\frac{c4^{(1+\alpha)n}}{R^{N\alpha}K^{2\alpha}}y_n^{1+\alpha}\ \ \mbox{if $R^{\frac{p(\ell-N)}{\ell(p-1)}}\||f|+\tau|u|\| _{\ell, B_R(x_0)}^{\frac{p}{p-1}}\leq K   $} .\nonumber
	\end{eqnarray}
Choose $K$ so large that
\begin{equation}\label{klb2}
	K\geq 2+R^{\frac{p(\ell-N)}{\ell(p-1)}}\||f|+\tau|u|\| _{\ell, B_R(x_0)}^{\frac{p}{p-1}}.
\end{equation}
Then we have
	\begin{eqnarray}
		y_{n+1}\leq\frac{c4^{\left(1+\alpha\right)n}}{\sigma^2R^{N\alpha}K^{2\alpha}}y_n^{1+\alpha}.\nonumber
	\end{eqnarray}
	By Proposition \ref{ynb}, if we further require $K$ to satisfy
	$$
	y_0\leq c\sigma^{\frac{2}{\alpha}}K^{2}R^{N},
	$$
	then
	\begin{equation}\label{non1}
	\sup_{B_{(1-\sigma)R}(x_0)}v\leq K.
	\end{equation}
	In view of \eqref{klb2}, it is enough for us to take
	\begin{eqnarray}
		K&=&\frac{c}{\sigma^{\frac{1}{\alpha}}}\left(\frac{y_0}{R^N}\right)^{\frac{1}{2}}+\left(R^{1-\frac{N}{\ell}}\||f|+\tau|u|\| _{\ell, B_R(x_0)}\right)^{\frac{p}{p-1}}+2.\label{non2}
	\end{eqnarray}
	Recall that 
	$$
	y_0=\int_{B_{R}(x_0)}\left[\left( v-\frac{K}{2}\right)^+\right]^2dx\leq \int_{B_{R}(x_0)}\left(w+\tau\right)^pdx\leq c\int_{B_{R}(x_0)}|\nabla u|^{2p}dx+c\tau^pR^N.
	$$
	Hence, by \eqref{non1} and \eqref{non2},
	\begin{eqnarray}
		\sup_{B_{(1-\sigma)R}(x_0)}|\nabla u|&\leq& \sup_{B_{(1-\sigma)R}(x_0)}v^{\frac{1}{p}}\leq K^{\frac{1}{p}}\nonumber\\
		&\leq& \frac{c}{\sigma^{\frac{1}{\alpha p}}}\left[\left(\avint_{B_{R}(x_0)}|\nabla u|^{2p}dx\right)^{\frac{1}{2p}}+\sqrt{\tau}\right]+\left(R^{1-\frac{N}{\ell}}\||f|+\tau|u|\| _{\ell, B_R(x_0)}\right)^{\frac{1}{p-1}}+c.\label{fr}
	\end{eqnarray}
	According to the interpolation inequality \eqref{linterp}, we have
	$$\|\nabla u\|_{2p, \br}\leq \ve\|\nabla u\|_{\infty,\br}+\frac{1}{\ve^{2p-1}}\|\nabla u\|_{1,\br},\ \ \ve>0.$$
	Use this in \eqref{fr} appropriately to get
	\begin{eqnarray}
		\sup_{B_{(1-\sigma)R}(x_0)}|\nabla u|&\leq& \ve\|\nabla u\|_{\infty,\br}+ \frac{c}{\ve^{2p-1}\sigma^{\frac{2}{\alpha }}}\avint_{B_{R}(x_0)}|\nabla u|dx\nonumber\\
		&&+\frac{c\sqrt{\tau}}{\sigma^{\frac{1}{\alpha p}}}+\left(R^{1-\frac{N}{\ell}}\||f|+\tau|u|\| _{\ell, B_R(x_0)}\right)^{\frac{1}{p-1}}+c.\nonumber
	\end{eqnarray}
Replace $R$ by $R_{n+1}\equiv R-\frac{R}{2^{n+2}}$ and take $\sigma=\frac{1}{2^{n+2}-1}$ to get
\begin{eqnarray}
	\sup_{\brn}|\nabla u|&\leq& \ve\sup_{\brno}|\nabla u|+ \frac{c2^{\frac{2n}{\alpha }}}{\ve^{2p-1}}\avint_{B_{R}(x_0)}|\nabla u|dx\nonumber\\
	&&+c\sqrt{\tau}2^{\frac{n}{\alpha p}}+\left(R^{1-\frac{N}{\ell}}\||f|+\tau|u|\| _{\ell, B_R(x_0)}\right)^{\frac{1}{p-1}}+c.\nonumber
\end{eqnarray}
 By iteration (\cite{D}, p.13),
 \begin{eqnarray}
 	\sup_{\brh}|\nabla u|&\leq&\ve^n\sup_{\brn}|\nabla u|+\left(\frac{c}{\ve^{2p-1}}\avint_{B_{R}(x_0)}|\nabla u|dx+c\sqrt{\tau}\right)\sum_{i=0}^{n-1}\left(2^{\frac{2}{\alpha }}\ve\right)^i\nonumber\\
 	&&+\left[\left(R^{1-\frac{N}{\ell}}\||f|+\tau|u|\| _{\ell, B_R(x_0)}\right)^{\frac{1}{p-1}}+c\right]\sum_{i=0}^{n-1}\ve^i.\nonumber
 \end{eqnarray}
Choose $\ve$ suitably small and then take $n\ra\infty$ to get
\begin{equation}
		\sup_{\brh}|\nabla u|\leq c\avint_{B_{R}(x_0)}|\nabla u|dx+c\sqrt{\tau}
+c\left(R^{1-\frac{N}{\ell}}\||f|+\tau|u|\| _{\ell, B_R(x_0)}\right)^{\frac{1}{p-1}}+c.\nonumber
\end{equation}
This completes the proof.
\end{proof} 
Further regularity results for solutions to equations of the p-Laplace type can be found in \cite{AZ,T} and the references therein. 

Given that
\begin{equation}\label{fcon}
f\in L^2(\Omega),\ \ \tau>0,	
\end{equation}
consider the boundary value problem
\begin{eqnarray}
	-\dest+\tau\ln\rho &=&f\ \  \mbox{in $\Omega$},\label{ae1}\\
	\nabla\rho\cdot\nu&=&0\ \ \mbox{on $\po$.}\label{ae2}
\end{eqnarray}
We refer the reader to \cite{BS,S} for the case $f\in L^1(\Omega)$. The main source of mathematical difficulty here is  the set where $\rho=0$.

A weak solution to \eqref{ae1}-\eqref{ae2} is a function $\rho\in W^{1,2}(\Omega)$ such that
\begin{eqnarray}
	\rho>0 \ \ \mbox{a.e. on $\Omega$},\ \ \ln\rho &\in& L^{2}(\Omega), \ \ \mbox{and}\nonumber\\
	\io\nabla\rho\nabla\varphi dx+\tau\io\ln\rho \varphi dx&=&\io f\varphi dx\ \ \mbox{for each $\varphi\in  W^{1,2}(\Omega).$}\nonumber
\end{eqnarray}

\begin{lem}\label{uex}
	Let \eqref{fcon} be satisfied. 
	Then there is a unique weak solution to \eqref{ae1}-\eqref{ae2}. Moreover, for each $\lambda\geq 1$ we have $\tau\ln\rho\in L^\lambda(\Omega)$ whenever $f\in L^\lambda(\Omega)$ and
	\begin{equation}\label{c21}
		\|\tau\ln\rho\|_\lambda\leq 	\|f\|_\lambda.
	\end{equation}
\end{lem}
\begin{proof} For the existence part, we consider the approximate problem
	\begin{eqnarray}
		-\Delta\rho_\delta+\delta\rho_\delta+\tau\psi_{\delta}(\rho_\delta)&=&f\ \  \mbox{in $\Omega$},\label{ae3}\\
		\nabla\rho_\delta\cdot\nu&=&0\ \ \mbox{on $\po$,}\label{ae4}
	\end{eqnarray}
	where $\delta\in (0,1)$ and
	\begin{equation}
		\psi_{\delta}(s)=
		\left\{\begin{array}{ll}
			\ln\left(s+\delta\right)	&\mbox{if $s> 0$,}\\
			\ln\delta	&\mbox{if $s\leq 0$.}
			\end{array}\right.\nonumber
	\end{equation} 
Existence of a solution to this problem can be established via the  Leray-Schauder Theorem. To see this, we define an operator $B$ from $L^2(\Omega)$ into itself as follows: We say $w=B(v)$ if $w$ solves problem
	\begin{eqnarray}
	-\Delta w+\delta w&=&f-\tau\psi_{\delta}(v)\ \  \mbox{in $\Omega$},\label{ae5}\\
	\nabla w\cdot\nu&=&0\ \ \mbox{on $\po$.}\label{ae6}
\end{eqnarray}
Note that $\psi_{\delta}(s)\geq \ln\delta$. Thus for $v\in L^2(\Omega)$ we have $\psi_{\delta}(v)\in L^q(\Omega)$ for each $q\geq 1$. Problem \eqref{ae5}-\eqref{ae6} has a unique weak solution $w$ in the space $ W^{1,2}(\Omega)$. That is, $B$ is well-defined. It is also easy for us to see that $B$ is continuous and maps bounded sets into compact ones. Now we verify (LS3) in Lemma \ref{lsf}. Suppose that $\sigma\in (0,1),\ w\in L^2(\Omega)$ are such that
$w=\sigma B(w) $, i.e.,
\begin{eqnarray}
	-\Delta w+
	\delta w&=&\sigma(f-\tau\psi_{\delta}(w))\ \  \mbox{in $\Omega$},\label{ae7}\\
	\nabla w\cdot\nu&=&0\ \ \mbox{on $\po$.}\label{ae8}
\end{eqnarray}
Use $w$ as a test function in \eqref{ae7} to get
\begin{eqnarray}
	\io|\nabla w|^2dx+\delta\io w^2dx&=&\sigma\io f wdx-\sigma\tau\io (\psi_{\delta}(w)-\psi_{\delta}(0)) wdx\nonumber\\
	&&-\sigma\tau \psi_{\delta}(0)\io wdx \nonumber\\
	&\leq &\frac{\delta}{2}\io w^2dx+\frac{1}{2\delta}\io(f -\tau\psi_{\delta}(0))^2dx.\label{non3}
\end{eqnarray}
Here we have used the fact that $\psi_{\delta}$ is an increasing function, and thus $(\psi_{\delta}(w)-\psi_{\delta}(0)) w\geq 0$.
Inequality \eqref{non3} implies
$$\io w^2dx\leq c(\delta).$$
Thus (LS3) in Lemma \ref{lsf} holds. As a result, problem \eqref{ae3}-\eqref{ae4} has a solution.

Next, we proceed to show that we can take $\delta\rightarrow 0$ in \eqref{ae3}-\eqref{ae4}.
To this end, we first write \eqref{ae3} in the form
\begin{equation}\label{ae11}
	-\Delta\rho_\delta+\delta(\rho_\delta-s_z)+\tau\psi_{\delta}(\rho_\delta)=f-\delta s_z\ \ \mbox{in $\Omega$,}
\end{equation}
where $$s_z=1-\delta$$
is the zero point of $\psi_{\delta}$, i.e., $\psi_{\delta}(s_z)=0$.  Introduce the function
\begin{equation}
	h_\varepsilon(s)=\left\{\begin{array}{ll}
		1&\mbox{if $s>\varepsilon$,}\\
		s&\mbox{if $|s|\leq \varepsilon$,}\\
		-1&\mbox{if $s<-\varepsilon$,}\ \ \ \varepsilon>0.
	\end{array}\right.\nonumber
\end{equation}
Obviously,
$$
\lim_{\varepsilon\rightarrow 0}=\mbox{sgn}_0(s)\equiv\left\{\begin{array}{ll}
	1&\mbox{if $s>0$,}\\
	0&\mbox{if $s=0$,}\\
	-1&\mbox{if $s<0$.}
\end{array}\right.
$$
We establish \eqref{c21} first. Let $\lambda\in(1,\infty)$ be given.
Then  the function $ (\tau|\psi_{\delta}(\rho_\delta)|+\varepsilon)^{\lambda-1}h_\varepsilon(\rho_\delta-s_z)$ is an increasing function of $\rho_\delta$. We easily check that it lies in $W^{1,2}(\Omega)$. Use it as a test function in \eqref{ae11}
to obtain
\begin{eqnarray}
	\lefteqn{(\lambda-1)\int_\Omega(\tau|\psi_{\delta}(\rho_\delta)|+\varepsilon)^{\lambda-2} \tau \mbox{sgn}_0(\psi_{\delta}(\rho_\delta))\psi_{\delta}^\prime(\rho_\delta)	h_\ve(\rho_\delta -s_z)|\nabla  \rho_\delta |^2\, dx}\nonumber\\
	&&+\int_\Omega( \tau|\psi_{\delta}(\rho_\delta)|+\varepsilon)^{\lambda-1} 	h_\varepsilon^\prime(\rho_\delta-s_z)|\nabla  \rho_\delta |^2\, dx	\nonumber\\
	&&+\delta\int_\Omega(\tau|\psi_{\delta}(\rho_\delta)|+\varepsilon)^{\lambda-1} 	h_\ve(\rho_\delta -s_z) (\rho_\delta-s_z) \, dx\nonumber\\
	&&+\io \tau\psi_{\delta}(\rho_\delta) (\tau|\psi_{\delta}(\rho_\delta)|+\varepsilon)^{\lambda-1}h_\varepsilon(\rho_\delta-s_z)dx\nonumber\\
	&=&\int_\Omega (f-\delta s_z)(\tau|\psi_{\delta}(\rho_\delta)|+\varepsilon)^{\lambda-1} 	h_\ve(\rho_\delta -s_z)  \, dx\leq \int_\Omega |f-\delta s_z|(\tau|\psi_{\delta}(\rho_\delta)|+\varepsilon)^{\lambda-1}  \, dx.\nonumber
\end{eqnarray}
Dropping the first three integrals in the above inequality and then letting $\varepsilon\rightarrow 0$ yield 
\begin{equation}
	\io|\tau\psi_{\delta}(\rho_\delta)|^\lambda\leq \io|f-\delta s_z||\tau\psi_{\delta}(\rho_\delta)|^{\lambda-1}dx\leq \|f-\delta s_z\|_{\lambda}\| \tau\psi_{\delta}(\rho_\delta) \|_{\lambda}^{\lambda-1},\nonumber
\end{equation}
from whence follows
\begin{equation}\label{flm}
	\|\tau\psi_{\delta}(\rho_\delta)\|_\lambda\leq \|f-\delta s_z\|_\lambda.
\end{equation}
Take $\lambda=1$ and make use of the definition of $\psi_{\delta}$ to derive
$$|\{\rho_\delta \leq 0\}|=\frac{1}{|\ln \delta|}\int_{\{\rho_\delta\leq 0\}}|\psi_{\delta}(\rho_\delta)|dx\leq \frac{c}{\tau|\ln \delta|}\ra 0\ \ \mbox{as $\delta\ra 0$}.$$
Remember that $|\{\rho_\delta \leq 0\}|+|\{\rho_\delta>0\}|=|\Omega|$. Thus we can invoke Lemma \ref{poin} to get
\begin{equation}\label{er1}
	\|\rho_\delta^-\|_2\leq \frac{c}{|\{\rho_\delta>0\}|^{\frac{1}{2}}}\|\nabla\rho_\delta\|_2\leq c\|\nabla\rho_\delta\|_2\ \ \mbox{at least for $\delta$ sufficiently small}.
\end{equation}
Use $\rho_\delta-s_z$ as a test function in \eqref{ae11} to get
\begin{eqnarray}
\lefteqn{	\io|\nabla\rho_\delta|^2dx+\delta\io(\rho_\delta-s_z)^2dx+\tau\io\psi_{\delta}(\rho_\delta)(\rho_\delta-s_z)dx}\nonumber\\
&=&\io (f  -\delta s_z)(\rho_\delta-s_z) dx\nonumber\\
	&\leq &\|f  -\delta s_z\|_{\frac{2N}{N+2}}\|\rho_\delta-s_z\|_{\frac{2N}{N-2}}\nonumber\\
	&\leq &c\|f  -\delta s_z\|_{\frac{2N}{N+2}}\left(\|\nabla\rho_\delta\|_2+\|\rho_\delta-s_z\|_1\right)\nonumber\\
	&\leq&\frac{1}{2}\|\nabla\rho_\delta\|_2^2+c\|f  -\delta s_z\|_{\frac{2N}{N+2}}^2+c\|f  -\delta s_z\|_{\frac{2N}{N+2}}\|\rho_\delta-s_z\|_1.\nonumber
\end{eqnarray}
Consequently,
\begin{eqnarray}
\lefteqn{\io|\nabla\rho_\delta|^2dx+	\io\psi_{\delta}(\rho_\delta)(\rho_\delta-s_z) dx}\nonumber\\
&\leq& c\|f  -\delta s_z\|_{\frac{2N}{N+2}}\|\rho_\delta-s_z\|_1+c\|f  -\delta s_z\|_{\frac{2N}{N+2}}^2\nonumber\\
&\leq &c\|\rho_\delta-s_z\|_1+c.\label{sa4}
\end{eqnarray}
Fix $L>1$. It follows from \eqref{er1} that
\begin{eqnarray}
\io|\rho_\delta|dx&\leq&\io\rho_\delta^-dx+\int_{\{\rho_\delta\geq L \}}(\rho_\delta-s_z)dx+c(L)\nonumber\\	
	&\leq &c\|\nabla\rho_\delta\|_2
	+\frac{1}{\ln(L+\delta)}\io\psi_{\delta}(\rho_\delta)(\rho_\delta-s_z) dx+c(L)\nonumber\\
	&\leq &\frac{1}{L}\io|\nabla\rho_\delta|^2dx+\frac{1}{\ln(L+\delta)}\io\psi_{\delta}(\rho_\delta)(\rho_\delta-s_z) dx+c(L).\nonumber
\end{eqnarray}
Use this in \eqref{sa4},  choose $L$ suitably large in the resulting inequality, and thereby obtain that
there exists a $\delta_0\in (0,1)$ such that
\begin{equation}\label{non4}
\|\nabla\rho_\delta\|_2+	\|\rho_\delta\|_1\leq c\ \ \mbox{for all $\delta\leq \delta_0$}.
\end{equation}
 We may assume that
\begin{equation}
	\rho_\delta\rightarrow \rho\ \ \mbox{weakly in $W^{1,2}(\Omega)$, strongly in $L^2(\Omega)$, and a.e. on $\Omega$.}\nonumber
\end{equation}
Take $\lambda=2$ in \eqref{flm} to get
\begin{eqnarray}
	\|\tau\psi_{\delta}(\rho_\delta)\|_2
	&\leq &\|f-\delta s_z\|_2\leq c.\label{rrr2}
\end{eqnarray}
It is easy to check that for a.e. $x$ in $\Omega$ we have
$$\psi_{\delta}(\rho_\delta(x))\ra\left\{ \begin{array}{ll}
-\infty&\ \ \mbox{if  $x\in\{\rho\leq 0\}$.}	\\
\ln\rho(x)&\mbox{ if $\rho(x)>0$.}
\end{array}\right.$$
By Fatou's lemma, we have
$$\int_{\{\rho\leq 0\}}\lim_{\delta\rightarrow 0}|\psi_{\delta}(\rho_\delta)|dx\leq \lim_{\delta\rightarrow 0}\io|\psi_{\delta}(\rho_\delta)|dx\leq c.$$ 
We must have
$$\left|\{\rho\leq 0\}\right|=0\ \ \mbox{and $\psi_{\delta}(\rho_\delta(x))\ra \ln\rho\ \ \mbox{a.e. in $\Omega$}$.}$$
This together with \eqref{rrr2} implies
$$\psi_{\delta}(\rho_\delta(x))\ra \ln\rho\ \ \mbox{weakly in $L^2(\Omega)$}.$$
We are ready to pass to the limit in \eqref{ae3}.

The uniqueness of a solution to \eqref{ae1}-\eqref{ae2} is trivial because $\ln\rt $ is strictly monotone. The proof is complete.
	\end{proof}

\section{Proof of the main theorem}

The proof of the main theorem will be divided into several claims. We begin by showing the existence of a solution to our approximate problem \eqref{ot1}-\eqref{ot2}. Once again, we assume that
$$1<p<2.$$
This necessitates the following claim.

\begin{clm}\label{c41}For each $\delta>0$ there is a solution $(\rho,u)$ to the problem
	\begin{eqnarray}
			-\Delta\rho+\tau\ln\rho  &=&f-au \ \ \ \mbox{in $\Omega$},\label{atom3}\\
		-\mdiv\left[F_\tau(|\nabla u |^2)\nabla u \right]-\delta\Delta u +\tau u  &=&\ln\rho  \ \ \ \mbox{in $\Omega$},\label{atom5}\\
		\nabla u \cdot\nu&=& 0\ \ \ \mbox{on $\partial\Omega$},\label{atom6}\\
		\nabla\rho \cdot\nu&=&0\ \ \ \mbox{on $\partial\Omega$}\label{atom4}
	\end{eqnarray}
in the space  $W^{1,2}(\Omega)\times W^{1,2}(\Omega) $.
\end{clm}
We would like to mention that if $p\geq 2$ this claim is not needed and we can directly prove an existence assertion for \eqref{ot1}-\eqref{ot2}. 

\begin{proof} 
		The existence assertion will be established via the Leray-Schauder Theorem. To do this, we define an operator $B$ from $ L^{ 2}(\Omega)$
	into itself as follows: For each $ v\in L^{ 2}(\Omega)$ we first solve the problem
	\begin{eqnarray}
		-\Delta\rho+\tau\ln\rho  &=&f-av\ \ \ \mbox{in $\Omega$},\label{om3}\\
		\nabla\rho\cdot\nu&=&0\ \ \ \mbox{on $\partial\Omega$}.\label{om4}
	\end{eqnarray}
	By Lemma \ref{uex}, there is a unique strong solution $\rho\in W^{1,2}(\Omega)$ with 
	\begin{equation}\label{hhh11}
		\|\tau\ln\rho\|_{ 2}\leq \|f-av\|_{ 2}.
	\end{equation}
From the classical Calder\'{o}n-Zygmund theorem we have $\rho\in W^{2,2}(\Omega)$.
 We use the function $\rho$ so obtained to form the problem 
	\begin{eqnarray}
		-\mdiv\left(F_\tau(|\nabla u|^2)\nabla u\right)-\delta\Delta u +\tau u &=&\ln\rho \ \ \ \mbox{in $\Omega$},\label{om5}\\
		\nabla u\cdot\nu&=& 0\ \ \ \mbox{on $\partial\Omega$}.\label{om6}
	\end{eqnarray}
Lemma \ref{l21} asserts
	that there is a unique weak solution $u\in W^{1,2}(\Omega)$ to \eqref{om5}-\eqref{om6}.	Note that the introduction of the term $-\delta\Delta u $ is to ensure that our solution $u$ does lie in the space $ W^{1,2}(\Omega)$. 
	In fact, we have
	\begin{equation}\label{hhh10}
		\delta\io|\nabla u|^2dx+	\tau\io u^2dx\leq c\io\ln^2\rho dx.
	\end{equation}
	We define
	$$
	B( v)=u.
	$$
	We can easily conclude that $B$ is well-defined.
	
	To see that $B$  maps bounded sets into precompact ones, we assume
$$
\mbox{$\{v_n\}$ is bounded  in $L^{ 2}(\Omega)$}.
$$
 Set $$u_n=B(v_n).$$
Then we have
\begin{eqnarray}
-\Delta\rho_n+\tau\ln\rho_n  &=&f-av_n\ \ \ \mbox{in $\Omega$},\label{tom3}\\
-\mdiv\left[F_\tau(|\nabla u_n|^2)\nabla u_n\right]-\delta\Delta u_n +\tau u_n &=&\ln\rho_n \ \ \ \mbox{in $\Omega$},\label{tom5}\\
\nabla u_n\cdot\nu&=& 0\ \ \ \mbox{on $\partial\Omega$},\label{tom6}\\
\nabla\rho_n\cdot\nu&=&0\ \ \ \mbox{on $\partial\Omega$}.\label{tom4}
\end{eqnarray}
According to \eqref{hhh11} and \eqref{hhh10}, we have
\begin{equation}\label{ta3}
	\|\tau\ln\rho_n \|_2\leq \|f-av_n\|_2\leq c,\ \  \|u_n\|_{W^{1,2}(\Omega)}\leq c \|\ln\rho_n \|_2\leq c.
\end{equation}
Consequently, $\{u_n\}$ is precompact in $L^2(\Omega)$.

Now we assume that
\begin{equation}
	v_n\ra v\ \ \mbox{strongly in $L^2(\Omega)$.}\nonumber
\end{equation}
Use $\rho_n-1$ as a test function in \eqref{tom3}, then employ an argument similar to the proof of \eqref{non4}, and thereby obtain
$$
\io|\nabla\rho_n|^2dx+\io|\rho_n|dx\leq c.
$$
With the aid of this and \eqref{ta3},
we may assume
\begin{eqnarray}
u_n&\rightarrow& u\ \ \mbox{weakly in $W^{1,2}(\Omega)$, strongly in $L^2(\Omega)$, and a.e. on $\Omega$},\label{hh10}\\	
\rho_n&\rightarrow&\rho\ \ \mbox{weakly in $W^{1,2}(\Omega)$, strongly in $L^2(\Omega)$, and a.e. on $\Omega$}.\nonumber
\end{eqnarray}
(Pass to subsequences if necessary.) It immediately follows that
\begin{equation}
\ln\rho_n\rightarrow\ln\rho\ \ \mbox{weakly in $L^{2}(\Omega)$}.\nonumber
\end{equation}
By \eqref{tom5}, we have
\begin{eqnarray}
\lefteqn{-\mdiv\left[F_\tau(|\nabla u_n|^2)\nabla u_n-F_\tau(|\nabla u_m|^2)\nabla u_m\right]}\nonumber\\
&&-\delta\Delta (u_n-u_m ) +\tau (u_n-u_m )=\ln\rho_n -\ln\rho_m\ \ \ \mbox{in $\Omega$}.\nonumber
\end{eqnarray}
Use $(u_n-u_m )$ as a test function in this equation to get
\begin{eqnarray}
\lefteqn{\io\left[F_\tau(|\nabla u_n|^2)\nabla u_n-F_\tau(|\nabla u_m|^2)\nabla u_m\right]\cdot\nabla(u_n-u_m )dx}	\nonumber\\
&&+\delta\io|\nabla(u_n-u_m )|^2dx+\tau \io(u_n-u_m )^2dx\nonumber\\
&=&\io(\ln\rho_n -\ln\rho_m)(u_n-u_m )dx\leq c\|u_n-u_m\|_2.\label{hh15}
\end{eqnarray}
This combined with \eqref{hh10} implies that
\begin{equation}\label{hh11}
	u_n\ra u \ \ \mbox{strongly in $W^{1,2}(\Omega)$.}
\end{equation}
Here we would like to remark that the introduction of the term $-\delta\Delta u$ in \eqref{atom5} has played a crucial role here. It guarantees the precompactness of $\{u_n\} $ in $L^2(\Omega)$. We can also estimate the first term in \eqref{hh15} from \eqref{uni} as follows:
\begin{eqnarray}
	\io|\nabla(u_n-u_m)|^pdx&= &\io\left(1+|\nabla u_m|^2+|\nabla u_n|^2\right)^{\frac{p(2-p)}{4}}\frac{|\nabla(u_n-u_m)|^p}{\left(1+|\nabla u_m|^2+|\nabla u_n|^2\right)^{\frac{p(2-p)}{4}}}dx\nonumber\\
	&\leq&\left(\io\left(1+|\nabla u_m|^2+|\nabla u_n|^2\right)^{\frac{p-2}{2}}|\nabla(u_n-u_m) |^2dx\right)^{\frac{p}{2}}\nonumber\\
	&&\cdot\left(\io\left(1+|\nabla u_m|^2+|\nabla u_n|^2\right)^{\frac{p}{2}}dx\right)^{1-\frac{p}{2}}\nonumber\\
	&\leq & c\left[\io \left[F_\tau(|\nabla u_n|^2)\nabla u_n-F_\tau(|\nabla u_m|^2)\nabla u_m\right]\cdot\nabla(u_n-u_m)dx\right]^{\frac{p}{2}}.\nonumber
\end{eqnarray} 
Of course, this estimate becomes redundant due to the second term in \eqref{hh15}. However, it becomes relevant when we take $\delta\ra 0$ \cite{X1}.

Returning to the proof of the claim, we can pass to the limit in \eqref{tom3}-\eqref{tom4}. The whole sequence $\{(\rho_n, u_n)\}$ converges because
the uniqueness assertion holds for both problem \eqref{om3}-\eqref{om4} and problem \eqref{om5}-\eqref{om6}. Thus, $B$ is continuous.

	We still need to show that there is a positive number $c$ such that 
\begin{equation}
	\|u\|_{2}\leq c\nonumber
\end{equation}
for all $u\in L^2(\Omega)$ and $\sigma\in (0,1]$ satisfying
$$u=\sigma B(u).$$
This equation is equivalent to the boundary value problem
\begin{eqnarray}
	-\dest+\tau \ln\rho &=&f-au\ \ \ \mbox{in $\Omega$},\label{ot9}\\
	-\mdiv\left(F_\tau\left(\left|\nabla \frac{u}{\sigma}\right|^2\right)\nabla u\right)-\delta\Delta u +\tau u&=&\sigma\ln\rho \ \ \ \mbox{in $\Omega$},\label{ot10}\\
	\nabla u\cdot\nu=\nabla\rho\cdot\nu&=& 0\ \ \ \mbox{on $\partial\Omega$}.
\end{eqnarray}
By the proof of Lemma \ref{uex}, $\rho$ is the limit of a sequence of functions whose lower bounds are positive. 
Thus we can use $\ln\rho $ as a test function in \eqref{ot9} to get
\begin{equation}\label{ta5}
	\tau\io\frac{|\nabla \rho|^2}{\rho }dx+\tau\io\ln^2\rho =\io f\ln\rho dx-a\io u\ln\rho dx.
\end{equation}
Using $u$ as a test function in \eqref{ot10} yields
\begin{equation}
	\sigma\io u\ln\rho dx=\tau\io u^2dx+\io F_\tau\left(\left|\nabla \frac{u}{\sigma}\right|^2\right)|\nabla u|^2dx+\delta\io|\nabla u|^2 dx\geq 0.\nonumber
\end{equation}
Substituting this into \eqref{ta5}, we obtain
$$
\io\ln^2\rho \leq c(\tau).
$$
Use $\tau u$ as a test function in \eqref{ot10} to derive
\begin{equation}
	\io(\tau u)^2dx\leq \io\ln^2\rho \leq c.\nonumber
\end{equation}
This completes the proof of Claim \ref{c41}.
\end{proof}	

\begin{clm}\label{p21}
	Let the assumptions of the main theorem hold.
	Then there is a weak solution $(\rho, u)$ to \eqref{ot1}-\eqref{ot2} in the space $W^{1,2}(\Omega)\times \left(W^{1,p}(\Omega)\cap L^{2}(\Omega) \right)$.
\end{clm}

This claim can be established by taking $\delta\ra 0 $ in \eqref{atom3}-\eqref{atom4}.
 The essential ingredients in doing so are already contained in the preceding proof. The only difference is that \eqref{hh11} is no longer true. However, for our purpose here it is enough for us to show that 
 \begin{equation}
 	\nabla u_\delta\ra\nabla u\ \ \mbox{a.e. on $\Omega$.}\nonumber
 \end{equation}
 This can be accomplished by invoking the proof of Claim 3.9 in \cite{X1}. We shall omit the details.
 
We would like to point out a mistake in the proof of Claim 3.2 in \cite{X1}. Specifically, the fact that $e^\psi\in W^{1,2}(\Omega)$ does not imply $\psi\in L^s(\Omega)$ for each $s>1$. It does only when $\psi$ is bounded below. Thus, the previous claim also serves as a correction to this mistake.

 We indicate the dependence of our approximate solutions on $\tau$ by writing
$$
	\rho=\rt,\ \ u=\ut.
$$
	Then problem \eqref{ot1}-\eqref{ot2} becomes
	\begin{eqnarray}
	-\Delta \rt+\tau\ln\rt &=&f-a \ut\ \ \ \mbox{in $\Omega$},\label{ot1t}\\
-\mdiv\left(F_\tau(|\nabla \ut|^2)\nabla \ut\right) +\tau \ut &=&\ln\rt  \ \ \ \mbox{in $\Omega$},\label{ot3t}\\
	\nabla \ut\cdot\nu&=&\nabla \rt\cdot\nu=0\ \ \ \mbox{on $\partial\Omega$}.\label{ot2t}
	\end{eqnarray}
We proceed to derive estimates for $\{\ut,\rt\}$ that are uniform in $\tau$.

\begin{clm}\label{c43} 
We have
	\begin{eqnarray}
	\io\left|\nabla\sqrt{\rt }\right|^2dx&\leq &c,\nonumber\\
	\|\ut\|_{W^{1.p}(\Omega)}&\leq &c.\label{uwp}
	\end{eqnarray}
\end{clm}
\begin{proof} 
Multiply through \eqref{ot3t} by $\tau$ and 
add the resulting equation to \eqref{ot1t} to get
\begin{equation}
	-\Delta \rt-\tau\mdiv\left(F_\tau(|\nabla \ut|^2)\nabla \ut\right)+(a+\tau^2) \ut =f.\nonumber
\end{equation}Integrate the above equation
over $\Omega$ to obtain
\begin{equation}\label{ha12}
\left|\io\ut dx\right|=\frac{1}{a+\tau^2}\left|\io fdx\right|\leq c.
\end{equation}
Use $\ln\rt $ as a test function in \eqref{ot1t} to get
\begin{eqnarray}
	\io\frac{|\nabla\rt|^2}{\rt }dx+\tau\io\ln^2\rt dx=  \io(f-a\ut)\ln\rt  dx.\label{ha13}
\end{eqnarray}
Use $\ut, f$ as  test functions in \eqref{ot3t} successively to get
\begin{eqnarray}
\io F_\tau(|\nabla\ut|^2)|\nabla\ut|^2dx+\tau\io\ut^2dx=\io\ut \ln\rt  dx,\nonumber\\
\io F_\tau(|\nabla\ut|^2)\nabla\ut\cdot\nabla fdx+\tau\io\ut fdx=\io f  \ln\rt  dx.\nonumber
\end{eqnarray}
Use the above two equations in \eqref{ha13} to deduce
\begin{eqnarray}
	\lefteqn{\io\frac{1}{\rt }|\nabla\rt|^2dx+\tau\io\ln^2\rt dx}\nonumber\\
	&&+a\io F_\tau(|\nabla\ut|^2)|\nabla\ut|^2dx+a\tau\io\ut^2\nonumber\\
	&= &\io F_\tau(|\nabla\ut|^2)\nabla\ut\cdot\nabla fdx+\tau\io\ut fdx.\label{non5}
\end{eqnarray}
Note that
\begin{eqnarray}
	\lefteqn{a\io|\nabla\ut|^pdx+a\beta_0\io |\nabla\ut|dx}\nonumber\\
	&\leq &a\io\left(|\nabla\ut|^2+\tau\right)^{\frac{p}{2}-1}(|\nabla\ut|^2+\tau)dx+	a\beta_0\io\left(|\nabla\ut|^2+\tau\right)^{-\frac{1}{2}}(|\nabla\ut|^2+\tau)dx\nonumber\\
&\leq&a\io F_\tau(|\nabla\ut|^2)|\nabla\ut|^2dx+a\tau^{\frac{p}{2}}|\Omega|+a\beta_0\tau^{\frac{1}{2}}|\Omega|.\label{ha15}
\end{eqnarray}
Here we have used the fact that $p\leq 2$. Use this again to get
\begin{eqnarray}
	\io \left|F_\tau(|\nabla\ut|^2)\nabla\ut\cdot\nabla f\right|dx&\leq &\io|\nabla\ut|^{p-1}|\nabla f|dx+\beta_0\io|\nabla f|dx\nonumber\\
	&\leq&\|\nabla\ut\|^{p-1}_p\|\nabla f\|_p+c\|\nabla f\|_1.\nonumber
\end{eqnarray}
Plug this and \eqref{ha15} into \eqref{non5} and apply Young's inequality \eqref{young} in the resulting inequality appropriately to derive
\begin{eqnarray}
	\lefteqn{\io\frac{1}{\rt }|\nabla\rt|^2dx+\tau\io\ln^2\rt dx}\nonumber\\
	&&+\io|\nabla\ut|^pdx+\tau\io|\nabla\ut|^2dx+\tau\io\ut^2dx\nonumber\\
	&\leq& c\io|\nabla f|^pdx+c\tau\io|\nabla f|^2dx+c\tau\io| f|^2dx+c\leq c.\label{ha17}
\end{eqnarray}By virtue of the Sobolev inequality and \eqref{ha12}, we have
\begin{eqnarray}
	\io|\ut|^pdx&\leq& 2^{p-1}\left(\io\left|\ut-\frac{1}{|\Omega|}\io \ut dx\right|^pdx+\frac{1}{|\Omega|^{p-1}}\left|\io \ut dx\right|^p \right)\nonumber\\
	&\leq&c\io|\nabla\ut|^pdx+c\leq c.\nonumber
\end{eqnarray}
This gives \eqref{uwp}.
The proof is complete.\end{proof}

\begin{clm}\label{c44}We have
	\begin{equation}
		\io|\ln\rt |dx\leq c.\nonumber
	\end{equation}
\end{clm}
This is Claim 3.6 in \cite{X1}. The proof there requires $p\leq 2$. Specifically,  (3.52) in \cite{X1} is not valid when $p>2$. Clearly, this restriction is not natural because formally large $p$ should make analysis easier. We will offer a proof here that removes the restriction.
\begin{proof}Denote by $\at$ the average of $\sqrt{\rt }$ over $\Omega$, i.e.,
	$$\at=\frac{1}{|\Omega|}\io\sqrt{\rt } dx.$$
It follows from the Sobolev embedding theorem that the sequence $\{\sqrt{\rt }-\at\}$ is bounded in $W^{1,2}(\Omega)$, and hence we can extract a subsequence, which we will not relabel, such that
\begin{equation}
	\sqrt{\rt }-\at \ \ \mbox{converges a.e. on $\Omega$.}\nonumber
\end{equation}
There are two possibilities: The sequence $\{\at\}$ is either bounded or unbounded. Assume the first case. We immediately yield
\begin{equation}
	\rt\rightarrow \rho \ \ \mbox{ a.e. on $\Omega$ (pass to a subsequence if need be) and $\rho(x)<\infty$  a.e. on $\Omega$.}\nonumber
\end{equation}
 Claim 3.5 in \cite{X1} asserts that for each $s\in [1, \frac{N}{N-2}]$ there is a positive number $c$
with the property
\begin{equation}\label{rr1}
	\io\rt^s dx\leq c. 
\end{equation}
	Integrate \eqref{ot3t} over $\Omega$ to get
$$
\io	\ln\rt dx=\tau\io\ut dx.
$$
This gives
\begin{equation}
	\io\ln^-\rt dx=\io\ln^+\rt dx-\tau\io\ut dx.\nonumber
\end{equation}
Consequently,
\begin{eqnarray}
	\io\left|\ln\rt \right|dx&=&\io\ln^-\rt +\io\ln^+\rt dx\nonumber\\
	&=&2\io\ln^+\rt dx-\tau\io\ut dx\nonumber\\
	&\leq &4\io\left(\sqrt{\rt }-1\right)^+dx+c\leq c.\label{r5}
\end{eqnarray}
The last step is due to \eqref{rr1} and \eqref{ha17}.

Now assume the second case holds. That is, 
\begin{equation}
	\lim_{\tau\rightarrow 0}a_\tau=\infty.\nonumber
\end{equation}
 We will show that this leads to contradiction, and hence it cannot occur.  Obviously, in this case, we have
\begin{equation}\label{rr3}
	\rt\ra\infty\ \ \mbox{a.e. on $\Omega$.}
\end{equation}	

Set
\begin{equation}
\ft=-f+a\ut^++\tau\ln^+\rt \ \ \mbox{and $\vt=\rt +\|\ft\|_{s}$},\ \ s>\max\left\{\frac{N}{2},\frac{N}{p}\right\}.	\nonumber
\end{equation}
Then we write \eqref{ot1t} as
\begin{equation}\label{efvt}
	-\Delta\vt=-\ft+a\ut^-+\tau\ln^-\rt \ \ \mbox{in $\Omega$.}
\end{equation}
Fix $y\in\overline{\Omega},\ r>0$.  Choose the cut-off function $\xi$ so that $\xi=1$ in $B_r(y)$, $\xi=0$ outside $B_{2r}(y)$, $0\leq\xi\leq 1$, and $|\nabla\xi|\leq \frac{c}{r}$. Use $\vt^{-1}\xi^2$ as a test function in \eqref{efvt}, take note of the fact that $\xi^2\nabla\vt\cdot\nu=0$ on $\partial(B_{2r}(y)\cap\Omega)$, and thereby obtain
\begin{eqnarray}
	\io\xi^2\vt^{-2}|\nabla\vt|^2 dx
&=&2\io\xi\nabla\xi\cdot\vt^{-1}\nabla\vt dx+\io\left(\ft-a\ut^--\tau\ln^-\rt \right)\vt^{-1}\xi^2 dx\nonumber\\
&\leq&\frac{1}{2}\io\xi^2\vt^{-2}|\nabla\vt|^2 dx+cr^{N-2}+\|\ft\|_{s}\|\vt^{-1}\xi^2\|_{\frac{s}{s-1}}.\label{r2}
\end{eqnarray}
Note that
\begin{equation}\label{non6}
\|\ft\|_{s}\|\vt^{-1}\xi^2\|_{\frac{s}{s-1}}=\left(\io\left(\frac{\|\ft\|_{s}\xi^2}{\rt +\|\ft\|_{s}}\right)^{\frac{s}{s-1}}dx\right)^{\frac{s-1}{s}}\leq c r^{\frac{N(s-1)}{s}}.
\end{equation}
Use this in \eqref{r2} to get
\begin{equation}\label{r7}
	\int_{B_{r}(y)\cap\Omega}|\nabla\ln\vt|^2\leq cr^{N-2}+cr^{N-2+\frac{N(s-1)}{s}-(N-2)}\leq cr^{N-2}.
\end{equation}
The last step is due to the fact that $s>\frac{N}{2}$.
Now we are in a position to apply the John-Nirenberg estimate (Theorem 7.21 in \cite{GT}). Upon doing so, we conclude that there is a positive number $\alpha$ depending only on $N,\Omega$, and the number $c$ in \eqref{r7} such that
\begin{equation}\label{jn}
	\io \vt^\alpha dx\io \vt^{-\alpha}dx\leq c.
\end{equation}
It is very important to remember that two functions have the same BMO norm if they only differ by a constant. Thus, \eqref{jn} holds in spite of \eqref{rr3}, which implies $\io\vt dx\ra \infty$. This is the key to the success of our argument.

Fix $\beta>1$. We use $\vt^{-\beta}$ as a test function in \eqref{efvt} to derive
\begin{equation}
	\beta\io\vt^{-\beta-1}|\nabla\vt|^2dx=\io\left(\ft-a\ut^--\tau\ln^-\rt \right)\vt^{-\beta}dx.\nonumber
\end{equation}
Consequently,
\begin{eqnarray}
	\io\left|\nabla\left(\vt^{-1}\right)^{\frac{\beta-1}{2}}\right|^2dx&=&\frac{(\beta-1)^2}{4\beta}\io\left(\ft-a\ut^--\tau\ln^-\rt \right)\vt^{-\beta}dx\nonumber\\
	&\leq& \frac{(\beta-1)^2}{4\beta}\|\ft\|_s\left\|\vt^{-\beta}\right\|_{\frac{s}{s-1}}\nonumber\\
	&=&\frac{(\beta-1)^2}{4\beta}\left\|\left(\vt^{-1}\right)^{\beta-1}\right\|_{\frac{s}{s-1}}.\nonumber
\end{eqnarray}
The last step is due to an algebraic manipulation similar to \eqref{non6}.
Combining this with the Sobolev embedding theorem, we deduce
\begin{eqnarray}
	\left\|\left(\vt^{-1}\right)^{\beta-1}\right\|_{\frac{N}{N-2}}&=&\left(\io \left(\vt^{-1}\right)^{\frac{\beta-1}{2}\frac{2N}{N-2}}dx\right)^{\frac{N-2}{2}}\nonumber\\
	&\leq &c\io\left|\nabla\left(\vt^{-1}\right)^{\frac{\beta-1}{2}}\right|^2dx+c\io\left(\vt^{-1}\right)^{\beta-1}dx\nonumber\\
	&\leq &c\beta \left\|\left(\vt^{-1}\right)^{\beta-1}\right\|_{\frac{s}{s-1}}.\nonumber
\end{eqnarray}
Take the $\frac{1}{\beta-1}$-th power of both sides to get
\begin{equation}
	\left\|\vt^{-1}\right\|_{\frac{N(\beta-1)}{N-2}}\leq c^{\frac{1}{\beta-1}}\beta^{\frac{1}{\beta-1}}\left\|\vt^{-1}\right\|_{\frac{s(\beta-1)}{s-1}}.\nonumber
\end{equation}
Set
$$\chi=\frac{N(s-1)}{s(N-2)}.$$
Then $\chi>1$ due to our assumption $s>\frac{N}{2}$. Take
$$\beta=1+\frac{\alpha(s-1)}{s}\chi^i,\ \ i=0,1,2,\cdots,$$
where $\alpha$ is given as in \eqref{jn}.
As a result, we have
\begin{eqnarray}
		\|\vt^{-1}\|_{\alpha\chi^{i+1}}&\leq& c^{\frac{s}{\alpha(s-1)\chi^i}}\left(1+\frac{\alpha(s-1)}{s}\chi^i\right)^{\frac{s}{\alpha(s-1)\chi^i}}\|\vt^{-1}\|_{\alpha\chi^{i}}\nonumber\\
		&\leq&c^{\frac{s}{\alpha(s-1)\chi^i}}\left(\chi^i+\frac{\alpha(s-1)}{s}\chi^i\right)^{\frac{s}{\alpha(s-1)\chi^i}}\|\vt^{-1}\|_{\alpha\chi^{i}}\nonumber\\
		&\leq &c^{\frac{s}{\alpha(s-1)\chi^i}}\chi^{\frac{si}{\alpha(s-1)\chi^i}}\|\vt^{-1}\|_{\alpha\chi^{i}}.\nonumber
\end{eqnarray}
Now we are ready to iterate $i$ to get
\begin{eqnarray}
	\|\vt^{-1}\|_{\alpha\chi^{i+1}}&\leq& c^{\frac{s}{\alpha(s-1)}\left(\frac{1}{\chi^i}+\frac{1}{\chi^{i-1}}+\cdots+1\right)}\chi^{\frac{s}{\alpha(s-1)}\left(\frac{i}{\chi^i}+\frac{i-1}{\chi^{i-1}}+\cdots+\frac{0}{1}\right)}\|\vt^{-1}\|_{\alpha}\nonumber\\
	&\leq &c\|\vt^{-1}\|_{\alpha}.\nonumber
\end{eqnarray}
Taking $i\ra\infty$ yields
\begin{equation}\label{r9}
	\inf_\Omega\vt\geq c\|\vt^{-1}\|_{\alpha}^{-1}\geq c\|\vt\|_{\alpha}\geq c\|\rt \|_{\alpha}.
\end{equation}
The second to last step is due to \eqref{jn}. 

By the definition of $\ft$,
\begin{equation}\label{r10}
	\|\ft\|_s\leq \|f\|_s+a\|\ut^+\|_s+\tau\|\ln^+\rt \|_s.
\end{equation}
We can easily infer from the proof of \eqref{of4} that
$$\|\ut^+\|_\infty\leq c\|\ut^+\|_1+c\left(\|\ln^+\rt \|_s\right)^{\frac{1}{p-1}}+\sqrt{\tau}\leq c\left(\|\ln^+\rt \|_s\right)^{\frac{1}{p-1}}+c.$$
Observe that the function $\ln^st$ is concave on the interval $[e^{s-1}, \infty)$. 
 Furthermore, Egoroff's theorem asserts that there is a positive number $c_0$ with the property
$$\left|\{\rt \geq e^{s-1}\}\right|\geq c_0\ \ \mbox{at least for $\tau$ sufficiently small.}$$
With these in mind, we derive from Jensen's inequality that
\begin{eqnarray}
	\|\ln^+\rt \|_s
&=& \left(\frac{1}{\alpha^s}\int_{\{\rt \geq e^{s-1}\}}\left(\ln\rt ^\alpha\right)^sdx+\int_{\{1\leq\rt < e^{s-1}\}}\ln^s\rt dx\right)^{\frac{1}{s}}\nonumber\\
&\leq&  \left(\frac{\left|\{\rt \geq e^{s-1}\}\right|}{\alpha^s}\ln^s\left[\frac{1}{\left|\{\rt \geq e^{s-1}\}\right|}\int_{\{\rt \geq e^{s-1}\}}\rt ^\alpha dx\right]dx+c\right)^{\frac{1}{s}}\nonumber\\
&\leq &c\ln\|\rt \|_\alpha+c.\nonumber
\end{eqnarray}
 Use this alone with \eqref{r10} in \eqref{r9} to derive
\begin{equation}\label{r4}
	\inf_\Omega\rt \geq c\|\rt \|_{\alpha}-\|f\|_s-c\left(\ln\|\rt \|_\alpha\right)^{\frac{1}{p-1}}-c\tau\ln\|\rt \|_\alpha-c.
\end{equation}
From \eqref{rr3}, with the aid of Fatou's lemma, we obtain
$$\lim_{\tau\rightarrow 0}\|\rt\|_{\alpha}=\infty.$$
This together with \eqref{r4} implies that there is a positive number $c_0$ such that
\begin{equation}
	\rt \geq c_0\ \ \mbox{for $\tau$ sufficiently small.}\nonumber
\end{equation}
It immediately follows that
$$\io\ln^-\rt dx\leq c.$$
By a calculation similar to \eqref{r5},
\begin{eqnarray}
	\io\left|\ln\rt \right|dx&=&\io\ln^-\rt +\io\ln^+\rt dx\nonumber\\
	&=&2\io\ln^-\rt dx+\tau\io\ut dx\leq c.\nonumber
\end{eqnarray}
By Fatou's lemma again, we have
\begin{equation}\label{lnb}
	\io|\ln\rho |dx=\io\lim_{\tau\rightarrow 0}|\ln\rt |dx\leq \limsup_{\tau\rightarrow 0}\io|\ln\rt |dx\leq c.
\end{equation}
This contradicts \eqref{rr3}.
The proof is complete.
	\end{proof}
\begin{clm}\label{aec} The sequence $\{\rt\}$ is precompact in $L^q(\Omega)$ for each $q\in [1,\frac{N}{N-2})$.
%
\end{clm}
\begin{proof}
For each $L>1$ we have from \eqref{lnb} that
\begin{eqnarray}
	\left|\left\{\rho\leq \frac{1}{L}\right\}\right|&\leq &\frac{1}{\ln L}\io|\ln\rho|dx\leq \frac{c}{\ln L},\nonumber\\
		\left|\left\{\rho>L\right\}\right|&\leq &\frac{1}{\ln L}\io|\ln\rho|dx\leq \frac{c}{\ln L}\nonumber
\end{eqnarray}
Consequently,
\begin{equation}\label{r41}
	|\{\rho =0\}|=|\{\rho =\infty\}|=0.
\end{equation}
The claim follows from Claim 3.5 in \cite{X1}.
\end{proof}


\begin{clm}\label{fini} There is an open subset $\Omega_0\subset\Omega$ such that
	 \begin{equation}\label{r40}
		\ln\rho\in L^\infty_{\textup{loc}}(\Omega_0)\ \ \mbox{and $\left|\Omega\setminus\Omega_0\right|=0$.}
	\end{equation}
\end{clm}
\begin{proof}
	Remember that $\rt$ satisfies
	\begin{equation}
		-\Delta \rt\geq-\ft\ \ \mbox{in $\Omega$.}\nonumber
	\end{equation}
	Lemma \ref{aec} asserts that the sequence $\{\ln^+\rt\}$ is bounded in $L^s(\Omega)$ for each $s\geq 1$, while Lemma \ref{l21} and \eqref{uwp} says that $\{\ut^+\}$ is bounded in $L^\infty(\Omega)\cap W^{1,p}(\Omega)$. Thus we may assume
	that
	\begin{eqnarray}
		\ut&\ra& u\ \ \mbox{weakly in $ W^{1,p}(\Omega)$ and a.e. on $\Omega$,}\nonumber\\
		\ut^+&\ra&u^+\ \ \mbox{weak$^*$ in $L^\infty(\Omega)$.}\nonumber
	\end{eqnarray}It immediately follows that
	\begin{equation}
		\ft=-f+a\ut^++\tau\ln^+\rt\ra-f+au^++\tau\ln^+\rho\equiv g\ \ \mbox{ strongly in $L^s(\Omega)$ for each $s\geq 1$.}\nonumber
	\end{equation} By Theorem 8.18 in (\cite{GT}, p.194), for each $s>\frac{N}{2}$ there exists a positive number c such that
	\begin{eqnarray}
		\inf_{\brh}\rt&\geq &\frac{c}{|\br|}\int_{B_{R}(x_0)}\rt dx-cR^{2\left(1-\frac{N}{2s}\right)}\|\ft\|_{s,\br}\nonumber\\
		&\ra&\frac{c}{|\br|}\int_{B_{R}(x_0)}\rho dx-cR^{2\left(1-\frac{N}{2s}\right)}\|g\|_{s,\br}\nonumber\\
		&=&cR^{2\left(1-\frac{N}{2s}\right)}\left(\frac{c}{R^{N+2\left(1-\frac{N}{2s}\right)}}\int_{B_{R}(x_0)}\rho dx-\|g\|_{s,\br}\right).\nonumber
	\end{eqnarray}
If $\limsup_{R\rightarrow 0}\frac{1}{R^{N+2\left(1-\frac{N}{2s}\right)}}\int_{B_{R}(x_0)}\rho dx>0$, then there is a positive number $R_0\in (0, \mbox{dist}(x_0, \po))$ such that
$$cR_0^{2\left(1-\frac{N}{2s}\right)}\left(\frac{c}{R_0^{N+2\left(1-\frac{N}{2s}\right)}}\int_{B_{R_0}(x_0)}\rho dx-\|g\|_{s,B_{R_0}(x_0)}\right)\equiv\ve_0>0.$$
Consequently,
\begin{equation}\label{rrr1}
		\inf_{B_{\frac{R_0}{2}}(x_0)}\rt\geq \frac{1}{2}\ve_0\ \ \mbox{for $\tau$ sufficiently small.}
\end{equation}
This together with Lemma \ref{l21} implies that 
$$\sup_{B_{\frac{R_0}{4}}(x_0)}\ut^-\leq c.$$
This puts us in a position to use Theorem 8.17 in \cite{GT}. Upon doing so, we arrive at
$$\sup_{B_{\frac{R_0}{8}}(x_0)}\rt\leq c.$$ It follows that $\ln\rho\in L^\infty(B_{\frac{R_0}{8}}(x_0))$. We take
$$\Omega_0=\cup_{s>\frac{N}{2}}\left\{x_0\in\Omega: \limsup_{R\rightarrow 0}\frac{1}{R^{N+2\left(1-\frac{N}{2s}\right)}}\int_{B_{R}(x_0)}\rho dx>0 \right\}.$$
Obviously, $\Omega_0$ is open and equal to the $\Omega_0$ in (D2). The second part of \eqref{r40} is a consequence of \eqref{r41}. The proof is complete.
	\end{proof}
\begin{clm}Let $\Omega_0$ be given as in Claim \ref{fini}. Then we have
	\begin{equation}
		|\nabla u|\in L^\infty_{\textup{loc}}(\Omega_0).\nonumber
	\end{equation}
	\end{clm}
\begin{proof}If \eqref{rrr1} holds, then $\{\ln\rt\}$ is bounded in $L^s\left(B_{\frac{R_0}{2}}(x_0)\right)$ for each $s>1$. A simple application of \eqref{nub} gives the desired result.
	\end{proof}
\begin{proof}[Proof of Theorem \ref{th1.1}]
	To complete the proof of the main theorem, we just need to mention that  (D3) can be established in the same way as in \cite{X2}, while the point-wise convergence of $\{\nabla\ut\}$ can be obtained as in \cite{X1}. This finishes the proof.
\end{proof}

We conclude by remarking that if $x_0$ is a singular point then  we must have 
$$\lim_{R\rightarrow 0}\frac{1}{R^{N+2\left(1-\frac{N}{2s}\right)}}\int_{B_{R}(x_0)}\rho dx=0,\ \ s>\frac{N}{2}.$$
If $p>\frac{N}{3}$, then $\frac{Np}{N-p}>\frac{N}{2}$. Thus we can apply Theorem 8.17 in \cite{GT} to get
$$\sup_{\brh}\rho\leq \frac{c}{|\br|}\int_{B_{R}(x_0)}\rho dx+cR^{2\left(1-\frac{N}{2s}\right)}\|f-au\|_{s,\br},\ \ s=\frac{Np}{N-p}.$$
This implies
$$\sup_{\br}\rho=o(R^{\frac{3p-N}{p}}).$$

\end{document}